\documentclass{amsart}
\usepackage{amsfonts,comment}
\usepackage{amsmath}
\usepackage{amssymb,amsmath,latexsym}
\numberwithin{equation}{section}

\def\<{\langle}
 \def\>{\rangle}

\def\e{{\epsilon}}

\def\w{\omega}
\newcommand{\be}{\begin{equation}}
\newcommand{\ee}{\end{equation}}
\newcommand{\bea}{\begin{eqnarray}}
\newcommand{\eea}{\end{eqnarray}}
\newcommand{\beas}{\begin{eqnarray*}}
\newcommand{\eeas}{\end{eqnarray*}}
\newtheorem{theorem}{Theorem}[section]
\newtheorem{Thm}{Theorem}[section]

\newtheorem{definition}[theorem]{Definition}
\newtheorem{Prop}[theorem]{Proposition}
\newtheorem{prop}[theorem]{Proposition}

\newtheorem{Cor}[theorem]{Corollary}

\newtheorem{lemma}[theorem]{Lemma}
\newtheorem{remark}[theorem]{Remark}
\newtheorem{example}[theorem]{Example}
\newtheorem{examples}[theorem]{Examples}
\newtheorem{foo}[theorem]{Remarks}
\newtheorem{As}{Assumption}

\newenvironment{Remark}{\begin{remark}\rm}{\end{remark}}

\newcommand{\q}{\quad} 
\newcommand{\td}{\mathrm{d}}

\def\Prod{\displaystyle\prod}
\def\F{\mathcal{F}}
\def\G{\mathcal{G}}

\def\te#1{\mathrm{e}^{#1}}

\def\mbb{\mathbb} \def\mbf{\mathbf} 
\def\mc{\mathcal} \def\unl{\underline} 
\renewcommand{\le}{\left}
\def\ri{\right}
\def\P{{\mathbb P}}
\def\E{{\mathbb E}}

\begin{document}

\title[Convergence of BS$\Delta$Es driven by random walks]{Convergence of BS$\Delta$Es driven by random walks to BSDEs:\\
the case of (in)finite activity jumps with general driver}

\author{Dilip Madan}
\address{Robert H. Smith School of Business,
University of Maryland} \email{dbm@rhsmith.umd.edu}

\author{Martijn Pistorius}
\address{Department of Mathematics, Imperial College London}
\email{m.pistorius@imperial.ac.uk}

\author{Mitja Stadje}
\address{Department of Econometrics and Operations Research,
University of Tilburg}
\email{m.a.stadje@uvt.nl}
\thanks{{\it Acknowledgements:}
MS acknowledges support by NWO VENI}

\begin{abstract}
In this paper we present a weak approximation scheme for
BSDEs driven by a Wiener process and an (in)finite
activity Poisson random measure with drivers that are general
Lipschitz functionals of the solution of the BSDE.
The approximating backward stochastic difference equations (BS$\Delta$Es) are  driven by random
walks that weakly approximate the given Wiener process and Poisson random measure.
We establish the weak convergence to the solution of the BSDE and
the numerical stability of the sequence of solutions of the BS$\Delta$Es.
By way of illustration we analyse explicitly a scheme with discrete step-size distributions.
\end{abstract}



\maketitle

\section{Introduction}
Backward stochastic differential equations (BSDEs) have turned up in a range of different setting,
notably in many applications in mathematical finance such as portfolio optimization and utility
indifference pricing, and also as non-linear expectations---see El  Karoui {\em et al.}~(1997)
for an overview of applications of BSDEs in finance and Delong~(2013) for a recent treatment of
the case of BSDEs with jumps. Unlike in the case of BSDEs without jumps, exact sampling methods from the probability distribution of the increments of the driving Poisson random measures are in general not readily available, which is an issue in the practical implementation of approximation schemes.
Motivated by this observation, we develop in this paper a weak approximation scheme for BSDEs driven by a Wiener process and independent Poisson random measure, allowing the approximating processes to be defined on filtrations that are different from the one the BSDE lives on. We also allow the drivers to take a general
Lipschitz-continuous functional form (see ~\eqref{bsde} below),
which is encountered in many applications.

\subsection*{Setting}
Let $T>0$ be a given horizon and let
$(\Omega,\F, {\mathbb P})$ be a probability space endowed with
a filtration $\mbf F = (\F_t)_{t\in[0,T]}$
generated by a $d_1$-dimensional Wiener process $W$ and an independent
$d_2$-dimensional L\'{e}vy
process ${X}$ ({\em i.e.}, a c\`{a}dl\`{a}g stochastic process with $X_0=0$ and stationary independent increments---refer to {\em e.g.}~Sato (1999) for background on L\'{e}vy processes). We assume that $X$ is a
zero-mean square-integrable process without Gaussian component, in which case $X$ is a pure-jump
martingale given by
\begin{equation}\label{jumps}
{X}_t=\int_{[0,t]\times \mathbb{R}^{d_2}\setminus\{0\}} x (N(\td s\times \td x)-\nu(\td x)\td s) =
\int_{[0,t]\times \mathbb{R}^{d_2}\setminus\{0\}} x
\tilde{N}(\td s\times \td x),\qquad t\in[0,T],
\end{equation} where $\nu$ denotes the L\'{e}vy measure of ${X}$,
$N$ is the Poisson random measure associated to the Poisson
point process $(\Delta{X}_t, t\in[0,T])$
of jumps of ${X}$ and
$\tilde N(\td s\times\td
x) = N(\td s\times\td x) -\nu(\td x)\td s$ is the corresponding
compensated Poisson random measure.
We consider in this paper BSDEs of the form
\begin{eqnarray}
&&\label{bsde} Y_t = F + \int_t^T f(s,Y_s,Z_s,\tilde{Z}_s)\td s -
 \int_t^T Z_s \td W_s-\int_{(t,
T]\times\mathbb{R}^{d_2}\setminus\{0\}}\tilde{Z}_s(x)
\tilde{N}(\td s\times \td x), \quad t\in[0,T], \\
\nonumber
&& \text{with driver function $f:[0,T]\times\mathbb{R}\times \mathbb{R}^{d_1}\times L^2(
\nu(\td x),\mc B(\mbb R^{d_2}\backslash\{0\}))\to\mathbb{R}$,}
\end{eqnarray}
for $\mathcal F_T$-measurable terminal conditions $F\in L^2({\mathbb P})$.
A triplet $(Y,Z,\tilde{Z})$ is called a solution of this BSDE
if ~\eqref{bsde} holds for all $t\in[0,T]$ and the triplet
takes values in
the product of the spaces $\mathcal{S}^2$, $\mathcal{H}^2$ and $\tilde{\mathcal{H}}^2$
of square-integrable $\mbf F$-adapted semi-martingales $Y$,
predictable processes $Z$ and $\mc P\otimes\mc B(\mbb R^{d_2}\setminus\{0\})$-measurable processes $\tilde Z$, respectively.\footnote{That is, these processes are square-integrable
with respect to
$$
|Y|_{\mathcal S^2}:=\E\le[\sup_{t\in[0,T]}|Y_t|^2\ri]^{1/2}, \q |Z|_{\mc H^2} :=
\E\le[\int_0^T |Z_t|^2\td t\ri]^{1/2}\q |\tilde Z|_{\tilde{\mc H}^2} :=
\E\le[\int_0^T\int_{\mbb R^{d_2}\backslash\{0\}} |\tilde Z_t(x)|^2\nu(\td x)\td t\ri]^{1/2}
$$
respectively, where $|\cdot|$ denotes the Euclidean norm. $\mathcal{P}$ denotes the predictable $\sigma$-algebra.} Under the standard setup, which we assume to be
in force throughout,
the driver $f$ is assumed to be such that (i)
$f$ is continuous as function of $t\in[0,T]$
at any $(y,z,\tilde z)$, and (ii) $f$ is
Lipschitz continuous in $(y,z, \tilde{z})$ uniformly for all $t\in [0,T]$, that
    is, there exists a positive $K$ satisfying
\begin{equation}\label{f}
|f(t,y_1,z_1,\tilde{z}_1)-f(t,y_0,z_0,\tilde{z}_0)|\leq
K\bigg(|y_1-y_0|+|z_1-z_0|+ \sqrt{\int_{\mathbb{R}^{d_2}\setminus\{0\}}
|\tilde{z}_1(x)-\tilde{z}_0(x)|^2\nu(\td x)}\bigg),
\end{equation}
for any $y_0, y_1\in\mbb R$,  $z_0, z_1\in\mbb R^{d_1}$ and $\tilde z_0, \tilde z_1\in L^2(
\nu(\td x),\mc B(\mbb R^{d_2}\backslash\{0\}))$. Under these conditions
 it is well-known that the BSDE \eqref{bsde} has a unique solution (see Tang \& Li (1994) and Royer (2006)).

\subsection*{Related literature}
BSDEs with jumps of the form in ~\eqref{jumps} play an
important role in many optimal control problems, see for
instance Tang \& Li~(1994), Eyraud-Loisel~(2005), Lim~(2006),
or Jeanblanc {\em et al.}~(2010).  Another main application of
BSDEs arises in utility maximization, see for instance El
Karoui \& Rouge~(2000), Hu, Imkeller \&\ M\"uller~(2005),
Kl\"oppel \& Schweizer (2007), and Sircar \&\ Sturm~(2011) in a
Brownian filtration. See
Mania \& Schweizer~(2005) and Morlais~(2009a) in a continuous
filtration, and Becherer~(2006) and Morlais~(2009b) in a
setting with finite jump activity, and Morlais~(2009b) and Pelsser \& Stadje~(2014) in a setting with infinite jump acitivity. Royer~(2006)
studied BSDEs driven by Brownian motion and a Poisson random
measure, and their application to $g$-expectations.
In the references quoted above the optimal solutions were characterized in terms of solutions of BSDEs, but
the problem of numerical approximation was not addressed in the case of BSDEs with jumps.

A common way to approximate a BSDE is by discretizing time,
replacing the BSDE by an appropriate discrete time backward
stochastic difference equation (BS$\Delta$E). We will consider
the sequence of BS$\Delta$Es driven by $d_1$-dimensional and
$d_2$-dimensional random walks $W^{(\pi)}$ and ${X}^{(\pi)}$ converging
to $W$ and ${X}$. In a setting without jumps, convergence
results for general random walks have
been obtained in Ma {\em et al.}~(2002), Cheridito \&\ Stadje
(2013), and in Briand {\em et al.}~(2001, 2002) using Picard
iteration arguments as well as results on convergence of
filtrations from Coquet {\em et al.}~(2000). While many authors
studied discrete schemes for the approximation of
solutions of BSDEs in a purely Brownian setting, in a setting
with jumps there is considerably less literature available. Lejay {\em et
al.}~(2007) is concerned with approximation schemes for BSDEs with one
single degenerate jump for a specific approximating process.
Contrary to the  references mentioned earlier in this paragraph
which took a random walk as the approximating process  Bouchard \&\ Elie~(2007)
considered numerical schemes in a pure finite activity jump setting (without a Brownian component) based on a direct discretization of the L\'{e}vy process. They showed convergence results for driver functions taking the form
$f(t,y,\int_{\mathbb{R}^{d_2}\setminus\{0\}}\rho(x)\tilde{z}(x)x
\nu(dx)),$ for a bounded functional $\rho$ and that for driver functions of this form
it suffices to compute (recursively backwards in time) the integral $\int_{\mathbb{R}^{d_2}\setminus\{0\}}\rho(x)\tilde{z}(x)x\nu(\td x)$.
Recently Aazizi~(2013) has extended the convergence results of Bouchard \&\ Elie (2007) to the setting of a forward-backward SDE driven by an infinite activity jump-process.

\subsection*{Contributions}In this paper we introduce a 
discrete-time scheme for the approximation of the solution of a BSDE
driven by a Wiener process and an independent Poisson random measure
allowing for a general Lipschitz-continuous driver function (where the driver may be a functional of
$\tilde{z}$). We prove $L^2$-stability and convergence results for the solutions to BS$\Delta$Es
generated by approximating random walks which may not be defined on the same filtration
as the continuous-time processes. The prime examples are finite (bi-, tri- and multinomial)
trees approximating the driving Brownian motion and L\'evy process.

Unlike the schemes considered in Bouchard \&
Elie~(2007) or Aazizi~(2013) the weak
approximation scheme considered in the current paper neither relies on the discrete process
being a discretization of the continuous-time process, nor on the discrete-time process being defined on the same filtration as the continuous-time process, nor on a Markovian structure.
In fact our results hold
for any suitable random walk and any terminal condition and are more in the spirit of Briand {\em
et al.}~(2002).  Furthermore, in all of the financial mathematics papers
quoted above the BSDEs in question take the more general form
given in (\ref{bsde}) (with the need to compute
the whole functional $\tilde{z}$) of which, up to this point, the
numerical implementation has received little attention in the
literature. One of the main contributions of the current paper
is to analyze this case. In particular, our results include the case of a driving L\'{e}vy process with infinite jump activity.
We show that when the probability of the random walks not moving is strictly positive
our BS$\Delta$Es satisfy strong $L^2$-regularity conditions which lead to stable numerical schemes. Note
that, in the infinite activity case, approximations schemes for L\'evy processes often exclude
or use a special technique to approximate the small jumps, see
for instance Asmussen \&\ Rosi\'nski (2001) and the reference therein for a discussion.

The outline of the proof of convergence is as follows.
We first prove convergence for terminal conditions and drivers satisfying
regularity and differentiability conditions on the underlying
Hilbert space. To overcome the difficulties arising from a
non-continuous limit we apply results from M\'emin (2003) concerning
the extended convergence of filtrations and use that the
solutions of the BS$\Delta$Es satisfy appropriate regularity
properties. The latter is shown by an induction over the Picard
sequences. General arguments on Hilbert spaces then conclude
the proof for smooth terminal conditions and drivers. For the
general case we deploy the $L^2$-regularity properties of the
solutions of BS$\Delta$Es mentioned above.

\subsection*{Contents} The remainder of the paper is structured as follows. First, in Section~\ref{secprel},
we present the random walk approximations and review the associated (extended) weak convergence results
that form the basis of the approximation schemes under consideration.
In Section~\ref{secBSDELTAE} we show numerical stability of the sequence of approximating BS$\Delta$Es
driven by these random walks, which forms an important step towards the main result, the convergence theorem,
that we present together with its proof in Section~\ref{secconv}. By way of illustration we present in Section~\ref{sec:example} an example in our setting of an explicit approximating BS$\Delta$E scheme driven by
a discrete random walk. Some proofs of auxiliary results are deferred to the appendix.

\section{Preliminaries}\label{secprel}
As approximation to the BSDE \eqref{bsde} we consider a sequence of discrete-time BSDEs (also referred to as BS$\Delta$Es, backward stochastic difference equations)
driven by processes with independent stationary increments $(W^{(\pi)}, X^{(\pi)})$
that are constant outside uniform time-grids $\pi$, with the collection of grids $\pi = \pi_{N}, {N\in\mbb N}$
given by $\pi_N :=\{t_0, t_1, \ldots, t_N\}$ with  $t_i= i T/N$, $i=0, \ldots, N$,
with mesh denoted by $\Delta = \Delta_N = T/N$. In the sequel we often write $\pi=\pi_N$ when no confusion is possible and identify the process $(W^{(\pi)}, X^{(\pi)})$ with the random walk
$(W^{(\pi)}_{t_i}, X^{(\pi)}_{t_i})_{t_i\in\pi}$.
In this section we specify these approximating random walks and collect weak-convergence results
that are deployed in the sequel.

\subsection{Random walk approximation}\label{ssecrw}
We assume that $W^{(\pi)}$ and $X^{(\pi)}$ are independent, square-intergrable
martingales defined on the probability space $(\Omega, \mc F^{(\pi)}, \mathbb P)$ which are piecewise constant on $[t_i,t_{i+1})$.
More specifically, let $W^{(\pi)} = (W^{(\pi),1}, \ldots, W^{(\pi),d_1})'$ (where $\prime$ denotes transpose) be a column-vector of zero-mean random walks that have independent stationary increments $\Delta W^{(\pi)}_{t_i} := W^{(\pi)}_{t_{i+1}} - W^{(\pi)}_{t_{i}}$
with second moment matching the corresponding second
moment of a Wiener process subject to a uniform moment-condition, {\em i.e.},
\begin{eqnarray}\label{W}
&& \E_{t_i}\le[\le(\Delta W^{(\pi)}_{t_i}\ri)\le(\Delta W^{(\pi)}_{t_i}\ri)'\ri] =
\Delta I_{d_1}, \qquad i=0, \ldots, N-1,\\
&& \sup_{\pi}\E[|W^{(\pi)}_T|^{2+\epsilon}] < \infty,\q\text{for some $\epsilon>0$},\label{W2}
\end{eqnarray}
where  $I_{d_1}$ the $d_1\times d_1$ identity matrix
and $\E_{t}[\cdot] = \E[\cdot|\mc F^{(\pi)}_{t}]$ for $t\in\pi$,
with $\mathbf F^{(\pi)} = (\mathcal F^{(\pi)}_{t}, t\in\pi)$ denoting the standard filtration generated by $(W^{(\pi)}, X^{(\pi)})$. The increments of $W^{(\pi)}$ may be for instance be taken to follow suitably chosen multivariate Bernoulli or Gaussian distributions.

Moreover, let $X^{(\pi)} = (X^{(\pi),1}, \ldots, X^{(\pi),d_2})'$
be a (column-vector of) zero-mean random walk with independent stationary increments $\Delta X^{(\pi)}_{t_i} := X^{(\pi)}_{t_{i+1}} - X^{(\pi)}_{t_{i}}$ satisfying the moment conditions
\begin{eqnarray}
\label{XX0}
&&  \Delta^{-1/2} \E[|\Delta X^{(\pi)}_{t_i}|] \longrightarrow 0, \q \Delta\to 0, \q\text{and} \\
&& \Delta ^{-1}\E_{t_i}\le[\le(\Delta X^{(\pi)}_{t_i}\ri)\le(\Delta X^{(\pi)}_{t_i}\ri)'\ri] \longrightarrow  \le( \nu_{k,l} \ri)_{k,l=1}^{d_2}, \qquad i=0, \ldots, N-1,\q\text{with}\q
\label{XX1}
\\
&& \nu_{k,l} = \int h_k(x) h_l(x) \nu(\td x),\qquad \text{$h_k(x)=x_k$,\ \  $k=1, \ldots, d_2$},
\ \text{and}
\nonumber\\
&& \sup_{\pi}\E[|X^{(\pi)}_T|^{2+\epsilon}] < \infty,\q\text{for some $\epsilon>0$}.\label{X2}
\end{eqnarray}
Note that ~\eqref{XX0} is satisfied when we take $\Delta X^{(\pi)}_{t_i}$
equal to the increment $X_{t_{i+1}} - X_{t_i}$ of $X$ over the interval $[t_i, t_{i+1}]$:
since $X$ is square-integrable by \eqref{X2}, the first absolute moment
of $X_t$ at small $t$ satisfies
$\E[|X_t|]=O(t)$ for $t\to 0$ (see Ludschgy \& Pag\`{e}s~(2008), Theorem 1).

It is also assumed that the step-size distribution $G^{(\pi)}$ satisfies
\begin{eqnarray}
\label{XX2}
&&   \int_{\mbb R^{d_2}\backslash\{0\}} g(x) \nu^{(\pi)}(\td x)
 \longrightarrow \int_{\mbb R^{d_2}\backslash\{0\}} g(x) \nu(\td x), \\
&&  \text{as $\Delta\to 0$, with}\q \nu^{(\pi)}(\td x):= \Delta^{-1} G^{(\pi)}(\td x),    \nonumber
\end{eqnarray}
for all continuous bounded functions $g:\mbb R^{d_2}\backslash\{0\}\to\mbb R$
that are 0 around $x=0$ and have a limit as $|x|\to\infty$.

Finally, we assume that
there is a positive probability that the random walk $X^{(\pi)}$ remains at the same location from one time-step to the next:
\begin{equation}\label{eq:zero}
 \liminf_{\Delta\to 0} \P\le(\Delta X^{(\pi)}_{t_i} = 0\ri) \ge a, \q \text{for some $a>0$.}
\end{equation}
Under condition \eqref{eq:zero} we establish that the corresponding sequence of
BS$\Delta$Es is numerically stable (see Theorem~\ref{thmstable}).
In the case that $X$ has finite activity  \eqref{eq:zero}
is naturally satisfied by the strong scheme $(X_{t_i})$ taking $a=\te{-\lambda}$ where $\lambda=\nu(\mbb R^{d_2}\backslash\{0\})$
denotes the jump rate. Thus, all strong schemes based on direct discretizations of
L\'evy processes with finite jump-activity that are in
$L^{2+\epsilon}$ satisfy all conditions specified above. While in the complementary case of infinite jump-activity
\eqref{eq:zero} is generally not satisfied by a strong discretisation scheme, this condition can be incorporated
in the construction of a weak scheme---see Section~\ref{sec:example} for explicit examples of weak schemes satisfying
\eqref{eq:zero} and all other conditions given above.

The conditions given in~\eqref{W}, \eqref{XX1} and \eqref{XX2}
are sufficient to guarantee functional weak convergence of the processes $(W^{(\pi)}, X^{(\pi)})$
to the L\'{e}vy process $(W,X)$ as the mesh size $\Delta$ tends to zero.
More precisely, as $\Delta\to 0$ we have
\begin{equation}\label{eqwc}
(W^{(\pi)}, X^{(\pi)}) \stackrel{\mc L}{\longrightarrow} (W,X),
\end{equation}
where $\stackrel{\mc L}{\longrightarrow}$ denotes weak-convergence
in the Skorokhod $J_1$-topology. This assertion follows as a direct consequence
of classical weak convergence theory (see Thm. VII.3.7 in Jacod \& Shiryaev (2003)),
given the conditions in~\eqref{W}, \eqref{XX1} and \eqref{XX2},
the independent increments property of
$(W^{(\pi)}, X^{(\pi)})$ and the independence of $W^{(\pi)}$ and $X^{(\pi)}$
on the one hand and that of $X$ and $W$ on the other hand.

In the sequel we assume that the random variables $(W^{(\pi)}, X^{(\pi)})_\pi$ have been defined
such that the convergence in \eqref{eqwc} holds in probability:
\begin{equation}\label{eqwcP}
(W^{(\pi)}, X^{(\pi)}) \stackrel{{\mbf P}}{\longrightarrow} (W,X),
\end{equation}
where $\stackrel{{\mbf P}}{\longrightarrow}$ denotes convergence in probability
in the Skorokhod $J_1$-topology\footnote{Such random variables can be constructed by the
Skorokhod embedding theorem}.
In the next results we collect  for later reference a number of ramifications of the convergence in ~\eqref{eqwcP}.

\begin{lemma}\label{sw}
{\rm (i)} Let $g:[0,T]\times \mbb R^{d_2}\to\mbb R$ be a continuous function
that is 0 in a neighbourhood of 0. Then
we have
\begin{eqnarray*}
&& \sum_{t_i\in\pi\backslash\{T\}\cap[0,\,\cdot\,]} g(t_i, \Delta X^{(\pi)}_{t_i})
\stackrel{{\mbf P}}{\longrightarrow} \int_{[0,\,\cdot\,]\times\mbb R^{d_2}\backslash\{0\}}g(s,x) N(\td s\times\td x),\\
&& \sum_{t_i\in\pi\backslash\{T\}\cap[0,\,\cdot\,]} \{g(t_i, \Delta X^{(\pi)}_{t_i}) - \E_{t_{i-1}}[g(t_i, \Delta X^{(\pi)}_{t_i})]\}
\stackrel{{\mbf P}}{\longrightarrow} \int_{[0,\,\cdot\,]\times\mbb R^{d_2}\backslash\{0\}}g(s,x) \tilde N(\td s\times\td x), \q\text{as $\Delta\to 0$}.
\end{eqnarray*}
The following limit holds in $L^1$:
\begin{equation}\label{eq:et0}
\lim_{\e\downarrow 0} \limsup_{\Delta\to 0} \sum_{j=1}^{T/\Delta}|\Delta X_{t_j}^{(\pi)}|^2 I_{\{|\Delta X_{t_j}^{(\pi)}|\leq\e\}} = 0.
\end{equation}

{\rm (ii)} Let $\bar{Z}:[0,T]\times\mathbb{R}^{d_2}\setminus\{0\}\to\mathbb{R}$
be a bounded function that is jointly continuous, and zero in a neighbourhood of zero,
and let the function $g^{(\pi)}_s(x)$
be piecewise constant (i.e. $g_s^{(\pi)} = g^{(\pi)}_{t_i}$ for $s\in[t_i, t_{i+1})$),
$(\F^{(\pi)}_{s}\otimes
\mathcal{B}(\mathbb{R}^{d_2}\setminus\{0\}))$-measurable for any $s\in[0,T]$,
and uniformly Lipschitz continuous as function of $x$
({\em i.e.}, for some constant $\hat{K}>0$ it holds
$\sup_{N\in\mbb N, s\in[0,T]}|g^{(\pi_N)}_{s}(x)|\leq \hat{K}|x|$ a.s.).
Then we have as $\Delta\to 0$
\begin{eqnarray}\label{C1}
\E\left[\sup_{i\in\{1, \ldots, N\}}\left|\sum_{j=0}^{i-1}
\int_{\mbb R^{d_2}\backslash\{0\}}g^{(\pi)}_{t_j}(x)^2\nu^{(\pi)}(\td x)\Delta - \int_{[0,t_i]\times\mbb R^{d_2}\backslash\{0\}}
g^{(\pi)}_{s}(x)^2
 \nu(\td x)\td s \ri|\ri]\to 0,\\
\E\left[\sup_{i\in\{1, \ldots, N\}}\left|\sum_{j=0}^{i-1}
\int_{\mbb R^{d_2}\backslash\{0\}}(\bar{Z}_{t_j} g^{(\pi)}_{t_j})(x)\nu^{(\pi)}(\td x)\Delta
- \int_{[0,t_i]\times\mbb R^{d_2}\backslash\{0\}}(\bar{Z}_{s}g^{(\pi)}_{s})(x)
 \nu(\td x)\td s
\ri|\ri]\to 0.\label{C2}
\end{eqnarray}
\end{lemma}
\begin{proof}
(i) The first relation is a direct consequence of the convergence in ~\eqref{eqwc}
and the fact that the map $\w \mapsto (\w, \sum_{s\leq \cdot} g(s,\Delta\w_s))$ (with $\Delta\w_s=\w_s-\w_{s-}$)
is continuous in the Skorokhod $J_1$-topology (see \cite[Cor. VI.2.8]{JS}).
The second relation follows from the first and the convergence in~\eqref{XX2}.  Finally, we turn to \eqref{eq:et0}.
Note that by (\ref{XX1})
\begin{equation}
\label{ge1}
\limsup_{\Delta\to 0}\E\left[ \sum_{j=1}^{T/\Delta}|\Delta X_{t_j}^{(\pi)}|^2\right]  = \limsup_{\Delta\to 0}\E[|X^{(\pi)}_T|^2]=\E[|X_T|^2].
\end{equation}
Furthermore, for any collection of continuous functions $(h_\e)_\e$ satisfying $I_{\{|x|>2\e\}}\leq |h_\e(x)|<I_{\{|x|>\e\}}$
the integrability conditions imply
\begin{equation}
\label{ge2}
\lim_{\e\downarrow 0} \liminf_{\Delta\to 0} \E \Big[\sum_{j=1}^{T/\Delta}|\Delta X_{t_j}^{(\pi)}|^2
I_{\{|\Delta X_{t_j}^{(\pi)})|>\e\}} \Big] \ge \lim_{\e\downarrow 0} \E\Big[\sum_{t:\Delta X_t\neq 0} |\Delta X_{t}|^2 h_\e(\Delta X_{t})\Big]=\E[|X_T|^2].
\end{equation}
The combination of (\ref{ge1}) and (\ref{ge2}) yields (\ref{eq:et0}).

(ii) For any $s\in[0,T]$ and $\e>0$,  the triangle inequality implies
\begin{eqnarray}\label{tride}
&& \left|\int_{\mbb R^{d_2}\backslash\{0\}}g^{(\pi)}_{s}(x)^2\nu^{(\pi)}(\td x) -
\int_{\mbb R^{d_2}\backslash\{0\}}
g^{(\pi)}_{s}(x)^2
 \nu(\td x)
\right| \leq \left|I^{(\pi)}(g^{(\pi)}_{s})\right| + J_\e^{(\pi)}, \q\text{with}\\
\nonumber
&& J_\e^{(\pi)} = \int_{\{|x|\leq \epsilon\}} g^{(\pi)}_{s}(x)^2 \nu^{(\pi)}(\td x) +
\int_{\{|x|\leq \epsilon\}} g^{(\pi)}_{s}(x)^2 \nu(\td x),
\end{eqnarray}
where, for any Borel-function $f\in L^2(\nu(\td x), \mathcal{B}(\mbb R^{d_2}\backslash\{0\}))\cap L^2(\nu^{(\pi)}(\td x),\mathcal{B}(\mbb R^{d_2}\backslash\{0\}))$, we denote
\begin{eqnarray}
&& I^{(\pi)}(f) = \int_{\{|x|>\epsilon\}}f(x)^2\nu^{(\pi)}(\td x) -
\int_{\{|x|>\epsilon\}}f(x)^2\nu(\td x).
\end{eqnarray}

Fix $\delta>0$ arbitrary and choose an $\epsilon>0$ from the set $\{a\in\mbb R^{(d_2}\backslash\{0\}:
\nu(\{x: |x|=|a|\})=0\}$
that satisfies
\begin{equation}\label{eq:deltae}
\hat K
\left(\int_{\{|x|\leq\e\}} |x|^2\nu^{(\pi)}(\td x) + \int_{\{|x|\leq\e\}} |x|^2\nu(\td x) \right)<\delta,
\end{equation}
uniformly over partitions $\pi$ [which is possible in
view of \eqref{eq:et0}].
Let us first show that $I^{(\pi)}(g^{(\pi)}_s)$ converges to zero in $L^1$ for any $s\in[0,T]$.
Let $X^{(\pi)}_{\epsilon}$ and $X_{\epsilon}$ be the pure-jump
processes induced by $X^{(\pi)}$ and $X$ by excluding all jumps smaller than $\e.$
Then $X_\e^{(\pi)}$ converges to $X_\e$ in the Skorokhod $J_1$-topology in probability as $\delta\to 0$.
Since the position
at the epoch of firt exit from a ball is a continuous path-functional in the Skorokhod $J_1$-topology
(see \cite[Prop. VI.2.12]{JS},
it follows in view of the integrability condtion \eqref{X2} that $X_\e^{(\pi)}(\tau_\e^{(\pi)})$ converges to
$X_\e(\tau_\e)$ in $L^2$, where $\tau_\e^{(\pi)} = \inf\{t\ge0: |X^{(\pi)}_{\e,t}|>\e\}$ and
$\tau_\e = \inf\{t\ge0: |X_{\e,t}|>\e\}$ are equal to the first-passage times into the complement
of the ball with radius $\e$.
The observation that $\tau_\e^{(\pi)}$ and
$\tau_\e$ are equal to the first time that
$X_\e^{(\pi)}$ and $X_\e$ jump in conjunction with the uniform Lipschitz-continuity of $g^{(\pi)}$,
\eqref{X2} and the fact $\nu^{(\pi)}(|x|>\e)\to \nu(|x|>\e)$
then imply
\begin{align*}
\lim_{\Delta\to 0}|I^{(\pi)}(g^{(\pi)})|&=\lim_{\Delta\to 0} \nu(|x|>\e)
\left|
\int_{\{|x|>\epsilon\}}g^{(\pi)}(x)^2\frac{\nu^{(\pi)}(\td x)}{\nu^{(\pi)}(|x|>\e)} -
\int_{\{|x|>\epsilon\}}g^{(\pi)}(x)^2\frac{\nu(\td x)}{\nu(|x|>\e)}\right|\\
&=
\lim_{\Delta\to 0}\nu(|x|>\e) \bigg|\E\Big[\big|g^{(\pi)}(X_\e^{(\pi)}(\tau_\e^{(\pi)}))\big|^2
-\big|g^{(\pi)}(X_\e(\tau_\e))\big|^2\Big]\bigg|=0.
\end{align*}
Furthermore, by the uniform Lipschitz-continuity of the function $g_s^{(\pi)}$ we also have that (a) the sequence $I^{(\pi)}(g_s^{(\pi)})$ is uniformly bounded and (b) $J_\e^{(\pi)}$ is bounded by the left-hand side of (\ref{eq:deltae}).
As a consequence, the bounded convergence theorem and the bounds \eqref{tride} and \eqref{eq:deltae}
imply that the limit as $\Delta\to 0$ of the left-hand side
of \eqref{C1} is smaller than $T \delta$. Since $\delta$ is arbitrary
the convergence stated in \eqref{C1} follows.

The proof of convergence in \eqref{C2} is analogous, and is omitted.
\end{proof}
\medskip

\subsection{Extended weak convergence}\label{ssecewc}
In order to establish the convergence of BSDEs
we also need to deploy the notions of extended weak convergence and
weak convergence of filtrations, the definitions of which, we recall from Coquet {\em et al.}~(2004) and
M\'emin (2003), are given as follows:

\begin{definition}\label{def:extco}\rm
Given stochastic processes $Z=(Z_t)_{t\in[0,T]}$ and $(Z^n)_{n\in\mbb N}$ with $Z^n=(Z^n_t)_{t\in[0,T]}$
defined on filtered probability spaces $(\Omega,\G,(\G_t),\mathbb{P})$ and
$(\Omega,\G^n,(\G^n_t),\mathbb{P})$ respectively, we say
(i) $\G^n$ weakly converges to $\G$ [denoted $\G^n\stackrel{w}{\to}\G]$
if for every $A\in \G$ the
sequence of processes $(\E[I_A|\mc G^n_t])_{t\in[0,T]}$ converges
to the process $(\E[I_A|\mc F_t])_{t\in[0,T]}$
and (ii) $(Z^n,\G^n)$ weakly converges to $(Z,\G)$ [denoted $(Z^n,\G^n)\stackrel{w}{\to}(Z,\G)]$
if for every $A\in \G$ the
sequence of processes $(Z^n_t,\E[I_A|\mc G^n_t])_{t\in[0,T]}$ converges
to the process $(Z_t,\E[I_A|\mc F_t])_{t\in[0,T]}$.
In both cases the convergence is in probability under the Skorokhod $J_1$-topology (on the space $D$ of c\`{a}dl\`{a}g functions).
\end{definition}
\begin{Remark}
\label{martconv}
It is clear that the notion of extended weak convergence in general is stronger than the
notion of weak convergence of filtration (see for instance Coquet {\em et al.}~(2004) and
M\'emin (2003) for a discussion).

However, in the notation of the previous definition, if $F^n_i$ converges to $F_i$ in $L^1$ for $i=1,\ldots,m$ and $\G^n\stackrel{w}{\to}\G$,
it may be shown by an application of Doob's maximal inequality (see Coquet {\em et al.}~(2004), Remark 1)
that we have the convergence $(\E[F^n_1|\G^n_\cdot],\ldots,\E[F^n_m|\G^n_\cdot]) \to (\E[F_1|\mc G_\cdot],\ldots,\E[F_m|\mc G_\cdot])$ in probability in the Skorokhod
$J_1$-topology. In particularly, if $\G^n$ converges to $\G$ weakly, $L^n$ is a $\G^n$-martingale and $L$ is a $\G$-martingale then $L^n_T\to L_T$ in $L^1$ implies that $(L^n,\G^n)\stackrel{w}{\to} (L,\G)$ in the extended sense, see also Proposition 7 in Coquet  {\em et al.}~(2004) or Proposition 1 in M\'emin (2003).
\end{Remark}
We recall (from Proposition 2 in M\'{e}min (2003)) that
$(W^{(\pi)}, X^{(\pi)})$ converges to the L\'{e}vy process $(W,X)$
in the sense of extended convergence, due to
the independence of the increments of the two-coordinate processes $W^{(\pi)}$
and $X^{(\pi)}$, in conjunction with the fact that
the filtration $\mc F^{(\pi)}$ is generated by the process $(W^{(\pi)}, X^{(\pi)})$:

\begin{Prop}[Proposition 2, M\'{e}min (2003)]\label{thm:ewx}
We have $((W^{(\pi)}, X^{(\pi)}), \mc F^{(\pi)}) \stackrel{w}{\to} ((W, X), \mc F)$ as $\Delta\to 0$.
In particular, $\mc F^{(\pi)}\stackrel{w}{\to}\mc F$.
\end{Prop}

If a sequence of square-integrable martingales converges to a limit in the sense of
extended convergence that is given above, the convergence
of the corresponding quadratic variation  and predictable compensator processes
also holds true, which is a fact that is deployed in the proof of convergence of BSDEs.

\begin{theorem}[Corollary 12, M\'emin (2003)]
\label{coro12} Let $(L^{(\pi)})$ be a sequence of square integrable
$\G^{(\pi)}$-measurable martingales, and let $L$ be a square
integrable quasi-left continuous $(\G_t)$-martingale.
If $L^{(\pi)}_T\to  L_T$ in $L^2$ and
 $(L^{(\pi)},\G^{(\pi)}) \stackrel{w}{\to}
(L,\G)$, then we have
$$\big(L^{(\pi)} ,[L^{(\pi)} , L^{(\pi)} ],\<L^{(\pi)} , L^{(\pi)} \>\big)\to \big( L,[L,L],\<L,L\>\big)$$ in probability under
the Skorokhod $J_1$-topology, where, for any square integrable martingale $M$, $[M, M]$ and $\<M,M\>$ denote the
associated quadratic variation and predictable compensator, respectively.
\end{theorem}

We record some results concerning the convergence of cross-variations which follow
as implications of Theorem~\ref{coro12} and are deployed later in the paper.

\begin{Cor}\label{cor:cv}Under the assumptions on the processes
$(L^{(\pi)})$ and $L$ in Theorem~\ref{coro12}, the following hold true:

{\rm (i)} As $\Delta\to 0$, $\<W^{(\pi)}, L^{(\pi)}\>\to \<W,L\>$, in probability in the Skorokhod $J_1$-topology.

{\rm (ii)} Assume that $\bar{Z}:[0,T]\times\mathbb{R}^{d_2}\to\mathbb{R}$
is bounded, jointly continuous, and zero in an
environment around zero, and consider the stochastic processes
$U^{(\pi)}=(U^{(\pi)}_t)_{t\in[0,T]}$ and $U=(U_t)_{t\in[0,T]}$ given by
$$
{U}^{(\pi)}_t:=
\sum_{t_i\in\pi\cap[0,t]} \{\bar Z_{t_i}(\Delta X^{(\pi)}_{t_i}) - \E_{t_{i-1}}[\bar Z_{t_i}(\Delta X^{(\pi)}_{t_i})]\}, \quad {U}_t:=\int_{[0,t]\times
\mathbb{R}^{d_2}\setminus\{0\}}\bar{Z}_s(x) \tilde{N}(\td s\times \td x).
$$
As $\Delta\to 0$, $\<U^{(\pi)}, L^{(\pi)}\>\to \<U,L\>$, in probability in the Skorokhod $J_1$-topology.
\end{Cor}
\begin{proof}
(i) Since $W^{(\pi)}$ $(L^{(\pi)})$ converges to $W$ ($L$, respectively)
in probability in the Skorokhod $J_1$-topology and $W$ is continuous,
this entails that joint processes $(W^{(\pi)}+ L^{(\pi)},W^{(\pi)}-L^{(\pi)})$
converge in probability in $J_1$ to $(W+L,W-L).$ By Remark \ref{martconv} and Proposition \ref{thm:ewx} this convergence holds true in the extended sense with the filtrations $\F^{(\pi)}$ and $\F$. Since $(W^{(\pi)}_T+ L^{(\pi)}_T,W^{(\pi)}_T-L^{(\pi)}_T)$
actually converges in $L^2$ to $(W+L,W-L)$ (by assumption for $L$ and by conditions \eqref{W} and \eqref{W2} for $W$),
we deduce from Theorem \ref{coro12} that
$(\<W^{(\pi)}+L^{(\pi)}\>,\<W^{(\pi)}-L^{(\pi)}\>)$ converges
to $(\<W+L\>,\<W-L\>)$ in probability in the Skorokhod $J_1$-topology.
As a consequence, we have
\begin{equation}\label{cv1}
\<W^{(\pi)}, L^{(\pi)}\> =
\frac{1}{4}\Big(\<W^{(\pi)}+L^{(\pi)}\>-\<W^{(\pi)}-L^{(\pi)}\>\Big)
\to \frac{1}{4}\Big(\<W+L\>-\<W-L\>\Big)=\<W,L\>
\end{equation}
in probability in the Skorokhod $J_1$-topology, as stated.

(ii) We start by noting (from Lemma~\ref{sw}) that as $\Delta\to 0$
$U^{(\pi)}$ converges to $U$  in probability in the Skorokhod $J_1$-topology.
As $\bar Z$ is bounded and zero in a neighbourhood of zero, it follows from (\ref{X2}) that the collection
$(U^{(\pi)})_{\pi}$ is bounded in $L^{2+\epsilon}$, so that in particular $U^{(\pi)}_T\to U_T$ in $L^2$.
Since the filtration satisfy $\mc F^{(\pi)}\stackrel{w}{\to}\mc F$ we have (by Proposition~\ref{thm:ewx} and Remark~\ref{martconv})
$$(U^{(\pi)}_{\cdot},L^{(\pi)}_{\cdot})= (\mathbb{ E}[U^{(\pi)}_T|\mc F^{(\pi)}_\cdot],
\mathbb{ E}[L^{(\pi)}_T|\mc F^{(\pi)}_\cdot])
\stackrel{\Delta\to 0}{\to}(\mathbb{
E}[U_T|\F_{\cdot}],\mathbb{
E}[L_T|\F_{\cdot}])=(U_{\cdot},L_{\cdot}).$$
By similar arguments as in part (i) it then follows that we have the convergence of
$\<U^{(\pi)}, L^{(\pi)}\>$ to $\<U,L\>$ in probability in the Skorokhod $J_1$-topology.
\end{proof}

\section{BS$\Delta$Es}\label{secBSDELTAE}
We turn next to the formulation of the approximating BS$\Delta$Es, the construction of their
solutions and numerical stability.

\subsection{Formulation}
Since by switching from the Wiener process $W$ to the process $W^{(\pi)}$ we lose the predictable representation property,
it is well known that
we need to include in the formulation of the BS$\Delta$E
an additional orthogonal martingale term $(M^{(\pi)})$, which thus leads us to the following
BS$\Delta$E on the grid $\pi$:
\begin{eqnarray}\nonumber
 Y^{(\pi)}_{t_i} &=& F^{(\pi)} + \sum_{j=i}^{N-1} f^{(\pi)}(t_j, Y^{(\pi)}_{t_j}, Z^{(\pi)}_{t_j},
\tilde Z^{(\pi)}_{t_j})\Delta -\sum_{j=i}^{N-1}Z^{(\pi)}_{t_j} \Delta W^{(\pi)}_{t_j} \\
&& - \sum_{j=i}^{N-1}\le\{\tilde Z^{(\pi)}_{t_j}(\Delta X^{(\pi)}_{t_j})I_{\{\Delta X^{(\pi)}_{t_j}\neq 0\}} -  \E_{t_j}\le[\tilde Z^{(\pi)}_{t_j}(\Delta X^{(\pi)}_{t_j})I_{\{\Delta X^{(\pi)}_{t_j}\neq 0\}}\ri]\ri\}
- \left(M^{(\pi)}_T -  M^{(\pi)}_{t_i}\right), \label{bsdelta}
\end{eqnarray}
 where the random variable $F^{(\pi)}\in L^2(\mathcal F^{(\pi)}_T)$ is the final condition,
and the driver $f^{(\pi)}:[0,T]\times\mbb R\times\mbb R^{d_1}\times L^2(\nu^{(\pi)},\mc B(\mbb R^{d_2}\backslash\{0\}))\to\mbb R$
is a function that is piecewise constant ({\em i.e.},
$f^{(\pi)}(s,\cdot)=f^{(\pi)}(t_i,\cdot)$ for $s\in[t_i, t_{i+1})$)
and is uniformly Lipschitz-continuous in $(y,z,\tilde z)$, {\em i.e.},
for some $K>0$ we have for all $t\in[0,T]$
\begin{equation}\label{KLip}
|f^{(\pi)}(t,y_1,z_1,\tilde z_1) - f^{(\pi)}(t,y_0,z_0,\tilde z_0)| \leq
K\le(|y_1-y_0| + |z_1-z_0| +  \sqrt{\E_{\nu^{(\pi)}}[ (\tilde z_1(\xi) - \tilde z_0(\xi))^2]}\ri),
\end{equation}
where, for any Borel-function $f$,  $\E_{\nu^{(\pi)}}[f(\xi)^2]:= \int f(z)^2 \nu^{(\pi)}(\td z)$.

A quadruple
$(Y^{(\pi)}, Z^{(\pi)}, \tilde Z^{(\pi)}, M^{(\pi)})$
is a solution of the BS$\Delta$E \eqref{bsdelta} if it satisfies \eqref{bsdelta} for all $t_i\in\pi$
where $Y^{(\pi)}_{t_i}$ and (the components of the row-vector) $Z^{(\pi)}_{t_i}$ are in $ L^2(\td \mathbb P,\mathcal F^{(\pi)}_{t_i})$, $\tilde Z^{(\pi)}_{t_i}$ lies in $L^2(G^{(\pi)}(\td x)\times\td \mathbb P, \mathcal B(\mbb R^{d_2}\backslash\{0\})\otimes \mc F_{t_{i}}^{(\pi)})$ and $M^{(\pi)}=(M^{(\pi)}_{t_i})$ is a zero-mean square-integrable
$\mbf F^{(\pi)}$-martingale on $\pi$ that
is orthogonal to $(W^{(\pi)}_{t_i})$ and to the martingales $(M^k_{t_i})$ with increments $\Delta M^k_{t_i} = k_{t_i}(\Delta X^{(\pi)}_{t_i}) -  \E_{t_i}\le[k_{t_i}(\Delta X^{(\pi)}_{t_i})\ri]$
for any function $(k_{t_i})_{t_i}$ with $k_{t_i}\in  L^2(G^{(\pi)}(\td x)\times\td \mathbb P, \mathcal B(\mbb R^{d_2})\otimes \mc F_{t_{i}}^{(\pi)})$.

The BS$\Delta$E can be equivalently expressed in differential notation as
\begin{eqnarray}\nonumber
\Delta Y^{(\pi)}_{t_i} &=& - f^{(\pi)}(t_i, Y^{(\pi)}_{t_i}, Z^{(\pi)}_{t_i},
\tilde Z^{(\pi)}_{t_i})\Delta +  Z^{(\pi)}_{t_i} \Delta W^{(\pi)}_{t_i}\\
&& +\  \le\{\tilde Z^{(\pi)}_{t_i}(\Delta X^{(\pi)}_{t_i})I_{\{\Delta X^{(\pi)}_{t_i}\neq 0\}} -  \E_{t_i}\le[\tilde Z^{(\pi)}_{t_i}(\Delta X^{(\pi)}_{t_i})I_{\{\Delta X^{(\pi)}_{t_i}\neq 0\}}\ri]\ri\}
 + \Delta M^{(\pi)}_{t_i}, \label{bsdelta2}\\
Y^{(\pi)}_T &=& F^{(\pi)},
\end{eqnarray}
where $i=0, \ldots, N-1$.
We have the following result:
\begin{Prop}\label{lem:BSDEsol} For $\Delta<1/K$
the BS$\Delta$E \eqref{bsdelta} has a unique solution
$(Y^{(\pi)}, Z^{(\pi)}, \tilde Z^{(\pi)}, M^{(\pi)})$,
which satisfies the relations: for $t_i\in\pi$,
\begin{eqnarray}\label{Y1}
Y_{t_i}^{(\pi)} &=& f^{(\pi)}(t_i, Y^{(\pi)}_{t_i}, Z^{(\pi)}_{t_i}, \tilde Z^{(\pi)}_{t_i})\Delta
+ \E_{t_i}[Y_{t_{i+1}}^{(\pi)}]\\ \label{Y2}
&=& \E_{t_i}\le[F^{(\pi)} +  \sum_{j=i}^{N-1} f^{(\pi)}(t_j, Y^{(\pi)}_{t_j}, Z^{(\pi)}_{t_j},
\tilde Z^{(\pi)}_{t_j})\Delta \ri],\\
 Z_{t_i}^{(\pi)} &=& \Delta^{-1}\ \E_{t_i}\le[Y^{(\pi)}_{t_{i+1}}\Delta W^{(\pi)}_{t_{i}} \ri],
\label{z1}\\
\tilde Z_{t_i}^{(\pi)}(x) &=& \E_{t_{i}}\le[Y^{(\pi)}_{t_{i+1}}|\Delta X^{(\pi)}_{t_i} = x\ri] - \E_{t_{i}}\le[Y^{(\pi)}_{t_{i+1}}|\Delta X^{(\pi)}_{t_i} = 0\ri],
\label{z2}
\\ \nonumber
\Delta M^{(\pi)}_{t_i} &=& Y_{t_{i+1}}^{(\pi)} - \E_{t_i}\le[Y_{t_{i+1}}^{(\pi)}\ri] -
 Z_{t_{i}}^{(\pi)} \Delta W^{(\pi)}_{t_{i}}\\
 &\phantom{=}&  - \le\{\tilde Z^{(\pi)}_{t_i}(\Delta X^{(\pi)}_{t_i})I_{\{\Delta X^{(\pi)}_{t_i}\neq 0\}}  -  \E_{t_i}\le[\tilde Z^{(\pi)}_{t_i}(\Delta X^{(\pi)}_{t_i})I_{\{\Delta X^{(\pi)}_{t_i}\neq 0\}} \ri]\ri\}. \label{m}
\end{eqnarray}
\end{Prop}
\begin{proof}
First of all we verify that a given solution
$(Y^{(\pi)}, Z^{(\pi)}, \tilde Z^{(\pi)}, M^{(\pi)})$ of the BS$\Delta$E \eqref{bsdelta2} satisfies the stated relations.
By taking conditional expectations with respect to $\mc F^{(\pi)}_{t_i}$ in \eqref{bsdelta} and ~\eqref{bsdelta2}
and using that the martingale increments $\Delta W^{(\pi)}_{t_i}$, $\tilde Z^{(\pi)}_{t_i}(\Delta X^{(\pi)}_{t_i}) -  \E_{t_i}\le[\tilde Z^{(\pi)}_{t_i}(\Delta X^{(\pi)}_{t_i})\ri]$
and $\Delta M^{(\pi)}_{t_i}$ are orthogonal and have zero mean
we find \eqref{Y1} and ~\eqref{Y2}. Similarly, multiplying the left- and right-hand sides of ~\eqref{bsdelta2} with the coordinates of the vector $\Delta W^{(\pi)}_{t_i}$ and subsequently taking the $\mc F^{(\pi)}_{t_i}$-conditional expectations yields ~\eqref{z1} in view of the moment condition in ~\eqref{W}.
Multiplying with an arbitrary function $g\in L^{\infty}(\mc F^{(\pi)}_{t_i}\otimes \mc B(\mbb R^{d_2}))$
and taking conditional expectations and using \eqref{Y1} shows denoting $A = \{\Delta X^{(\pi)}_{t_i}\neq 0\}$
\begin{equation}\label{detzt}
\E_{t_i}\le[\le\{Y^{(\pi)}_{t_{i+1}} - \E_{t_i}[Y^{(\pi)}_{t_{i+1}}]\ri\}g(\Delta X^{(\pi)}_{t_i})\ri] =
\E_{t_i}\le[\le\{\tilde Z^{(\pi)}_{t_i}(\Delta X^{(\pi)}_{t_i})I_{A} - \E_{t_i}[\tilde Z^{(\pi)}_{t_i}(\Delta X^{(\pi)}_{t_i})I_A]\ri\}g(\Delta X^{(\pi)}_{t_i})\ri],
\end{equation}
which implies $I_A \tilde Z^{(\pi)}_{t_i}(\Delta X^{(\pi)}_{t_i}) = C + \E_{t_i}[Y_{t_{i+1}}|\Delta X^{(\pi)}_{t_i}]$
for some $C\in L^2(\mc F^{(\pi)}_{t_i})$. By inserting this expression into \eqref{detzt} and taking $g(x)=I_{\{0\}}(x)$
we find with $A^c = \{\Delta X^{(\pi)}_{t_i}= 0\}$
$$
- \le(C + \E_{t_i}[Y^{(\pi)}_{t_{i+1}}]\ri)\E_{t_i}[I_{A^c}]  =
\E_{t_i}[Y^{(\pi)}_{t_{i+1}}I_{A^c}] - \E_{t_i}[Y^{(\pi)}_{t_{i+1}}]\E_{t_i}[I_{A^c}]
\Rightarrow C =  -\E_{t_i}\le[Y^{(\pi)}_{t_{i+1}}\bigg|{A^c}\ri],
$$
which implies that we have ~\eqref{z2}. The relation \eqref{m} directly follows by combining~\eqref{bsdelta2} and \eqref{Y1}.

Next we verify existence. Define the quadruple $(Y^{(\pi)}, Z^{(\pi)}, \tilde Z^{(\pi)}, M^{(\pi)})$
by the right-hand sides of \eqref{Y1}, \eqref{z1}, \eqref{z2} and \eqref{m}. Note that $Y^{(\pi)}$ is determined
uniquely by the implicit equation \eqref{Y1} (since the map $\Psi: L^2(\td\P,\mc F^{(\pi)}_{t_i}) \to L^2(\td\P,\mc F^{(\pi)}_{t_i})$ given by $\Psi(Y) = f^{(\pi)}(t_i, Y, Z^{(\pi)}_{t_i}, \tilde Z^{(\pi)}_{t_i})\Delta
+ \E_{t_i}[Y_{t_{i+1}}^{(\pi)}]$ is a contraction in case $K\Delta < 1$ as a consequence of the Lipschitz condition
\eqref{KLip}). Furthermore, it is straightforward to verify that the measurability and integrability requirements are satisfied, as well as \eqref{bsdelta2}.

Finally, we verifty the orthogonality of the martingale $M^{(\pi)}$.
To see that  $M^{(\pi)}$ and $W^{(\pi)}$ are orthogonal, we note that
since $\le\{\tilde Z^{(\pi)}_{\cdot}(\Delta X^{(\pi)}_{\cdot})I_{\{\Delta X^{(\pi)}_{\cdot}\neq 0\}}  -  \E_{\cdot}\le[\tilde Z^{(\pi)}_{\cdot}(\Delta X^{(\pi)}_{\cdot})I_{\{\Delta X^{(\pi)}_{\cdot}\neq 0\}} \ri]\ri\}$ and $\Delta W^{(\pi)}_\cdot$ are orthogonal, we have
by definition of $Z_{t_i}^{(\pi)}$ and $\Delta W^{(\pi)}_{t_i}$
$$
\E_{t_i}[\Delta M^{(\pi)}_{t_i}\Delta W^{(\pi)}_{t_i}] =
\E_{t_i}[Y^{(\pi)}_{t_{i+1}}\Delta W^{(\pi)}_{t_i}] - \E_{t_i}[( Z^{(\pi)}_{t_{i}}
 \Delta W^{(\pi)}_{t_i}) \Delta W^{(\pi)}_{t_i}] = 0.
$$
Furthermore, for any function $k_{t_i}\in L^{\infty}(\mc F_{t_i}^{(\pi)}\otimes\mc B(\mbb R^{d_2}))$
it holds
\begin{eqnarray}
\nonumber
\lefteqn{\E_{t_i}[\Delta M^{(\pi)}_{t_i}\{k_{t_i}(\Delta X_{t_i}^{(\pi)}) -
\E_{t_i}[k_{t_i}(\Delta X_{t_i}^{(\pi)})]\}]}\\ \nonumber
&=&
\E_{t_i}[ Y^{(\pi)}_{t_{i+1}}k_{t_i}(\Delta X_{t_i}^{(\pi)})] -
 \E_{t_i}[ Y^{(\pi)}_{t_{i+1}}]\E_{t_i}[k_{t_i}(\Delta X_{t_i}^{(\pi)})]\\
\nonumber
&\phantom{=}& - \E_{t_i}[\tilde Z^{(\pi)}_{t_i}(\Delta X_{t_i}^{(\pi)}) k_{t_i}(\Delta X_{t_i}^{(\pi)})]
+ \E_{t_i}[\tilde Z^{(\pi)}_{t_i}(\Delta X_{t_i}^{(\pi)})]
\E_{t_i}[k_{t_i}(\Delta X_{t_i}^{(\pi)})] = 0,
\label{eq:orth}
\end{eqnarray}
where we used that $\tilde Z^{(\pi)}_{t_i}(0)=0$, inserted the form \eqref{z2} and used the tower-property of conditional expectation. Hence,
$M^{(\pi)}$ is orthogonal to the martingales with increments
$k_{t_i}(\Delta X_{t_i}^{(\pi)}) -
\E_{t_i}[k_{t_i}(\Delta X_{t_i}^{(\pi)})]$, and the proof is complete.
\end{proof}

In the case that the final value $F^{(\pi)}$ is independent of $W^{(\pi)}$
the orthogonal martingale $M^{(\pi)}$ vanishes.
\begin{Prop}
\label{mn}
If $F^{(\pi)}$ is independent of $W^{(\pi)}$ then $M^{(\pi)}\equiv 0.$
\end{Prop}
In particular, it follows that
in the pure jump case, the martingale $M^{(\pi)}$ is zero and
the representation property holds true.
\begin{proof} The assertion follows directly from (\ref{m})
once we have shown that in
the case that $F^{(\pi)}\in L^2(\F^{(\pi)}_{t_{i+1}})$ is independent of $W^{(\pi)}$
then $Z^{(\pi)}_{t_i}$ and $\tilde Z^{(\pi)}_{t_i}$ defined in \eqref{z1} and \eqref{z2}
are such that $Z^{(\pi)}_{t_i} = 0$ and
\begin{equation}\label{rep}
F^{(\pi)}=\E_{t_i}[F^{(\pi)}] + \le\{\tilde Z^{(\pi)}_{t_i}(\Delta X^{(\pi)}_{t_i}) -  \E_{t_i}\le[\tilde Z^{(\pi)}_{t_i}(\Delta X^{(\pi)}_{t_i})\ri]\ri\}.
\end{equation}
That $Z^{(\pi)}_{t_i} = 0$ follows directly from \eqref{z1} [with $[Y^{(\pi)}_{t_{i+1}}=F^{(\pi)}$],
since $W^{(\pi)}_{t_i}$ has zero mean and is independent of  $F$.
To see that the identity \eqref{rep} holds we first note that, as $F^{(\pi)}\in L^2(\td \mathbb P, \F^{(\pi)}_{t_{i+1}})$
and $F^{(\pi)}$ is independent of $W^{(\pi)}$
there exists a function $f$ in  $L^2(G^{(\pi)}(\td x)\times\td \mathbb P, \mathcal B(\mbb R^{d_2})\otimes \mc F_{t_{i}}^{(\pi)})$ satisfying $F^{(\pi)} = f(\Delta X^{(\pi)}_{t_i})$.
Inserting the forms of $F^{(\pi)}$ and $\tilde Z^{(\pi)}_{t_i}$ in the rhs of \eqref{rep} and
performing straightforward manipulations  (similar to those in the proof of Proposition \ref{lem:BSDEsol})
shows that the rhs and lhs in \eqref{rep} coincide.
\end{proof}

\subsection{Numerical stability}

In this section we turn to the numerical stability of the BS$\Delta$Es in $L^2$ sense.
We start by specifying uniform conditions for the collection of drivers $(f^{(\pi)})$
of the BS$\Delta$Es.

\begin{As}\label{A2}\rm
(i) For some $K>0$,
the drivers $f^{(\pi)}$ are uniformly $K$-Lipschitz continuous ({\em i.e.}, $f^{(\pi)}$ satisfies \eqref{KLip}).

(ii) $f^{(\pi)}(t,0,0,0)$ is bounded uniformly over all $t\in\pi$ and partitions $\pi$.

(iii) For every $(t,y,z)\in[0,T]\times\mbb R\times\mbb R^{d_1}$
and uniformly Lipschitz continuous function $\tilde z$ ({\em i.e.}, $\tilde z$ for which
$|\tilde z(x)|/|x|$ is bounded over all $x\in\mbb R^{d_2}\backslash\{0\}$),
we have
\begin{equation}\label{eq:limd0}
\lim_{\Delta\to0}f^{(\pi)}(t,y,z,\tilde z) = f(t,y,z,\tilde z).
\end{equation}
\end{As}
\begin{foo}\rm
(i) Note that the functions $f^{(\pi)}(t,y,z,\tilde z)$ in \eqref{eq:limd0} are well-defined
since every Lipschitz continuous function $\tilde z$ is square-integrable with respect
to the measures $\nu^{(\pi)}$ and $\nu$.

(ii) In Assumption~\ref{A2}\,(iii) it suffices to require the convergence
of the drivers only for uniformly Lipschitz continuous functions $\tilde z$ as these functions
form a dense subset in $L^2(\nu^{(\pi)}, \mc B(\mbb R^{d_2}\backslash\{0\}))$.

(iii) When the driver $f(t,y,z,\cdot)$ is distribution-invariant under the measure
$\nu(\td x),$ {\em i.e.}, there exists a function $\hat{f}$ such that
$f(t,y,z,\tilde{z})=\hat{f}(t,y,z,\nu\circ \tilde{z}^{-1})$,
a natural first candidate for $f^{(\pi)}$ would be to set
$f^{(\pi)}(t,y,z,\tilde{z}):=\hat{f}(t,y,z,\nu^{(\pi)}\circ
\tilde{z}^{-1})$.
\end{foo}

We have the following estimate for  BS$\Delta$Es as in ~\eqref{bsdelta} with
drivers $f^{(\pi),0}, f^{(\pi),1}$ and
terminal conditions $F^{(\pi),0}, F^{(\pi),1}$
and corresponding solution quadruples denoted by $(Y^{(\pi),k}, Z^{(\pi),k}, \tilde Z^{(\pi),k}, M^{(\pi),k})$,
$k=0,1$, respectively.

\begin{theorem}\label{thmstable}
There exists an $n_0\in\mbb N$ and a constant $\bar C$ such that for all $\pi=\pi_N$ with $N\ge n_0$,
all drivers $f^{(\pi),0}, f^{(\pi),1}$ satisfying Assumption 1(i)-(ii), and
square integrable terminal conditions $F^{(\pi),0}, F^{(\pi),1}$,
and $t_i\in\pi$, we have
\begin{multline}
\E\le[ \max_{t_j\leq t_i, t_j\in\pi}|\delta Y^{(\pi)}_{t_j}|^2
+ \sum_{j=0}^{i-1}\le\{ |\delta Z^{(\pi)}_{t_j}|^2 \Delta
+ |\delta M^{(\pi)}_{t_j}|^2 + |\delta\tilde Z^{(\pi)}_{t_j}(\Delta X^{(\pi)}_{t_j}) -
\E_{t_j}[\delta\tilde Z^{(\pi)}_{t_j}(\Delta X^{(\pi)}_{t_j})]|^2
\ri\}\ri]\\
\leq \bar C \E\le[ |\delta Y^{(\pi)}_{t_i}|^2 + \sum_{j=0}^{i-1} |\delta f^{(\pi)}(t_j, Y^{(\pi),0}_{t_j},
Z^{(\pi),0}_{t_j}, \tilde Z^{(\pi),0}_{t_j})|^2\Delta\ri],\label{eq:stable}
\end{multline}
with $\delta Y^{(\pi)} = Y^{(\pi),0} - Y^{(\pi),1}$, etc.
\end{theorem}
\begin{Remark}\label{remstable}
In continuous-time the following analogous estimate holds true for some constant $\bar c>0$:
 \begin{eqnarray}
 \label{2stable}
\lefteqn{\mathbb{E}\left[\sup_{0\leq t\leq t'}|\delta  Y_t|^2  +
\int_0^{t'} |\delta Z_s|^2 \td s+\int_{[0,t'] \times \mathbb{R}^{d_2}\setminus\{0\}}|\delta \tilde{Z}_s(x)|^2 \nu(\td x)\td s \right]}\\
&\leq& \bar c \,\mathbb{E}\bigg[|\delta Y_{t'}|^2
+\int_0^{t'} |\delta f(s,Y^0_s,Z^{0}_s,\tilde{Z}^{0}_s)|^2\td s\bigg], \qquad t'\in[0,T].\nonumber
 \end{eqnarray}
For a proof of (\ref{2stable}), see for instance to Proposition 3.3 in
Becherer (2006) or Lemma 3.1.1 in Delong~(2013).
\end{Remark}
In the proof of Theorem \ref{thmstable}, which is provided in the Appendix,
the following estimate is deployed which is a consequence of the zero-jump-condition \eqref{eq:zero}:

\begin{lemma}
\label{finitejumps} There exist $\delta_0>0$ and $C'>0$ such that
for all $\Delta\leq \delta_0$, for all functions
$(\tilde{U}_{t_j})_j$ with $\tilde U_{t_j}(0)=0$ and $\tilde{U}_{t_j}\in
L^2(\nu^{(\pi)}(\td x)\times
\td\mathbb{P},\mathcal{B}(\mathbb{R}^{d_2})
\otimes\F^{(\pi)}_{t_j})$, and for any $j=0, \ldots, n-1$
we have
\begin{equation}\label{eq:Cp}
\sum_{i=j}^{n-1}
\Big(\E_{t_i}\le[|\tilde{U}_{t_i}(\Delta
X^{(\pi)}_{t_i})|^2\ri] - \Big|\E_{t_i}\le[\tilde{U}_{t_i}(\Delta
X^n_{t_i})\ri]\Big|^2\Big)
\ge
C' \sum_{i=j}^{n-1}
\Big|\E_{t_i}\le[\tilde{U}_{t_i}(\Delta
X^{(\pi)}_{t_i})\ri]\Big|^2.
\end{equation}
\end{lemma}

\begin{proof}
Assume without loss of generality that $j=0.$
Using H\"older's inequality we have
\begin{align}\nonumber
 \sum_{i=0}^{n-1}
\Big|\E_{t_i}\le[\tilde{U}_{t_i}(\Delta
X^{(\pi)}_{t_i})\ri]\Big|^2 &=
 \sum_{i=0}^{n-1}
\Big|\E_{t_i}\le[\tilde{U}_{t_i}(\Delta
X^{(\pi)}_{t_i})I_{\{\Delta
X^{(\pi)}_{t_i} \neq 0\}}\ri]\Big|^2\\
&\leq \Big(\max_i \mathbb{P}[\Delta
X^{(\pi)}_{t_i}\ne 0]\Big)  \sum_{i=0}^{n-1}
\E_{t_i}\le[\Big|\tilde{U}_{t_i}(\Delta
X^{(\pi)}_{t_i})\Big|^2\ri].
\label{b}
\end{align}
Since $X^{(\pi)}$ has stationary increments the
first factor in the final line is equal to
$\mathbb{P}[\Delta
X^{(\pi)}_{t_1}\ne 0]$,
which is bounded above by $(1-a+\delta)$
for all partitions with mesh $\Delta\leq \delta_0$,
where  $\delta$ is some number small enough such
that $a-\delta>0,$ and $\delta_0$ is chosen sufficiently small
using \eqref{eq:zero}. By combining the upper bound with \eqref{b}
we obtain \eqref{eq:Cp} (with $C'=a-\delta$).
\end{proof}

\subsection{Solution of the BS$\Delta$E via Picard iteration}\label{ssecpicard}

The process $(Y^{(\pi)}, Z^{(\pi)}, \tilde Z^{(\pi)})$
satisfying the BS$\Delta$E can be obtained as the limit of
an recursively defined Picard sequence
$(Y^{(\pi,p)},
Z^{(\pi,p)},
\tilde Z^{(\pi,p)})_{p\in\mbb N^*}$, which is initialised with
 $(Y^{(\pi,0)},
Z^{(\pi,0)},
\tilde Z^{(\pi,0)}) \equiv (0,0,0)$
 and is defined  for $p\in\mbb N$ and $t_i\in\pi$ by the right-hand sides of
formulas \eqref{Y2}, \eqref{z1} and \eqref{z2} respectively,
with  $Y^{(\pi)}_{t_j}$, $Z^{(\pi)}_{t_j}$ and $\tilde Z^{(\pi)}_{t_j}$ replaced by
 $Y^{(\pi,p-1)}_{t_j}$, $Z^{(\pi,p-1)}_{t_j}$ and $\tilde Z^{(\pi,p-1)}_{t_j}$.
We may associate to the sequence $(Y^{(\pi),p},
Z^{(\pi),p}, \tilde Z^{(\pi),p})_{p\in\mbb N^*}$ a sequence of square-integrable orthogonal martingales
$(M^{(\pi),p})_{p\in\mbb N^*}$ defined by
$M^{(\pi),0}\equiv 0$ and for $p\in\mbb N$ by $M^{(\pi),p} = \{M^{(\pi),p}_{t_i}, t_i\in\pi\}$ with
\begin{eqnarray*}
\Delta M^{(\pi),p}_{t_i} =
Y_{t_{i+1}}^{(\pi),p} - \E_{t_i}\le[Y_{t_{i+1}}^{(\pi),p}\ri] -
Z_{t_{i}}^{(\pi),p}\Delta W^{(\pi)}_{t_{i}}
 \ -\  \le\{\tilde Z^{(\pi),p}_{t_i}(\Delta X^{(\pi)}_{t_i}) -  \E_{t_i}\le[\tilde Z^{(\pi),p}_{t_i}(\Delta X^{(\pi)}_{t_i})\ri]\ri\}.
\end{eqnarray*}
We also note that we have
\begin{eqnarray} \label{fwrep}
&&\q m^{(\pi)}_{t_i} := \E_{t_i}\le[F^{(\pi)}   + \sum_{t_j\in\pi} f^{(\pi)}(t_j, Y_{t_j}^{(\pi),p},
Z_{t_j}^{(\pi),p}, \tilde Z_{t_j}^{(\pi),p})\Delta\ri]\\
&=& Y_0^{(\pi),p+1} + \sum_{t_j\in\pi, j< i} Z_{t_j}^{(\pi),p+1}\Delta W^{(\pi)}_{t_j}
+ \sum_{t_j\in\pi, j< i} \le\{\tilde Z_{t_j}^{(\pi),p+1}(\Delta X^{(\pi)}_{t_j}) -
\E_{t_i}\le[\tilde Z_{t_j}^{(\pi),p+1}(\Delta X^{(\pi)}_{t_j})\ri]\ri\}
+ M^{(\pi),p+1}_{t_i}.\nonumber
\end{eqnarray}
It is well-known that, as $p$ tends to infinity, the Picard sequence $(Y^{(\pi,p)},
Z^{(\pi,p)}, \tilde Z^{(\pi,p)}, M^{(\pi,p)})$
converges to $(Y^{(\pi)}, Z^{(\pi)},
\tilde Z^{(\pi)}, M^{(\pi)})$ .
In particular, it follows from Theorem~\ref{thmstable}
(by reasoning analogously as in Corollary 10 in Briand {\it et al.} (2002))
that for some $n_0\in\mbb N$ it holds
\begin{multline}\label{cor:picbsdde}
\sup_{\pi_N: N\ge n_0}
\E\bigg[
\sup_{t_i\in\pi_N}|Y^{(\pi)}_{t_i} - Y^{(\pi),p}_{t_i}|^2 + \sum_{t_i\in\pi_N}\{|Z^{(\pi)}_{t_i} - Z^{(\pi),p}_{t_i}|^2\Delta + \Delta(M^{(\pi)} - M^{(\pi),p})^2_{t_i}\} \\
\le. + \sum_{t_i\in\pi_N}
\le\{\tilde Z^{(\pi)}_{t_i}(\Delta X^{(\pi)}_{t_i}) - \tilde Z^{(\pi),p}_{t_i}(\Delta X^{(\pi)}_{t_i})
- \E_{t_i}\le[\tilde Z^{(\pi)}_{t_i}(\Delta X^{(\pi)}_{t_i}) - \tilde Z^{(\pi),p}_{t_i}(\Delta X^{(\pi)}_{t_i})\ri]
\ri\}^2
\ri]  \to 0 \q\text{as $p\to\infty$}.
\end{multline}

\section{Convergence}\label{secconv}
With the results concerning the convergence of the approximating random walks
and the properties of the discrete time BSDEs in hand, we turn next to the question
of weak convergence of BS$\Delta$Es to the limiting BSDE as the mesh size tends to zero.
Let $Y^{(\pi)}_t=Y^{(\pi)}_{t_i}$ for $t_i\leq t<t_{i+1}$ and define $(Z^{(\pi)}_t,\tilde{Z}^{(\pi)}_t,M^{(\pi)}_t)$ similarly.
\begin{Thm}\label{thm}
Let $(\pi)$ be a sequence of partitions $\pi$ with the mesh $\Delta$ tending to zero.
If $F^{(\pi)}$ converges to $F$ in $L^2$, then $Y^{(\pi)}\stackrel{\mathcal L}{\longrightarrow} Y$ and in particular
$$
Y_0^{(\pi)} \to Y_0.
$$
Moreover, with $d_S$ denoting the Skorokhod metric, we have
$$
\E[d_S^2(Y^{(\pi)},Y)] \to 0.
$$
\end{Thm}
\begin{proof}
The idea, inspired by Briand~{\em et al.}~(2001,2002),
is to reduce
the question of weak convergence of the solutions of the BS$\Delta$Es to the solution of BSDE
to that of the Picard sequences
by using the fact that both the solutions of the BSDE and
of the BS$\Delta$Es are equal to limits of Picard sequences.

Define the sequence $(Y^{\infty,p}, Z^{\infty,p}, \tilde Z^{\infty,p})_{p\in\mbb N\cup\{0\}}$
recursively by $(Y^{\infty,0}, Z^{\infty,0}, \tilde Z^{\infty,0})=(0,0,0)$ and
\begin{eqnarray*}
Y^{\infty,p+1}_t := F + \int_t^T
f(s,Y^{\infty,p}_{s},Z^{\infty,p}_s,\tilde{Z}^{\infty,p}_s)\td s -\int_t^T Z^{\infty,p+1}_s \td W_s
-\int_{(t,T]\times \mathbb{R}^{d_2}\setminus\{0\}}\tilde{Z}^{\infty,p+1}_s (x)\tilde{N}(\td s\times \td x)
\end{eqnarray*}
for $p\in\mbb N\cup\{0\}$,
where $(Z^{\infty,p+1}, \tilde Z^{\infty,p+1})$ are the unique coefficients in
the martingale representation of the square-integrable martingale $N^p=\{N^p_t, t\in[0,T]\}$:
\begin{eqnarray}\nonumber
N^p_t &:=& \E\le[ F + \int_0^T
f(s,Y^{\infty,p}_{s-},Z^{\infty,p}_s,\tilde{Z}^{\infty,p}_s)\td s
\bigg|\mc F_t\ri] - \E\le[ F + \int_0^T
f(s,Y^{\infty,p}_{s-},Z^{\infty,p}_s,\tilde{Z}^{\infty,p}_s)\td s
\ri]\\
 &=& \int_0^t Z^{\infty,p+1}_s \td W_s +\int_0^t
\tilde{Z}^{\infty,p+1}_s \tilde{N}(\td s\times \td x).
\label{eq:Np}
\end{eqnarray}
Furthermore, recall that we denote
by $(Y^{(\pi),p}, Z^{(\pi),p}, \tilde Z^{(\pi),p}, M^{(\pi),p})_{p\in\mbb N\cup\{0\}}$
the Picard sequences corresponding to the BS$\Delta$Es defined on the grid $\pi$. In the remainder of the proof
we will deploy the continuous-time extensions of $(Y^{(\pi),p}, Z^{(\pi),p}, \tilde Z^{(\pi),p}, M^{(\pi),p})_{p\in\mbb N\cup\{0\}}$ defined by taking paths to be piecewise constant; we
denote these extensions also by $(Y^{(\pi),p}, Z^{(\pi),p}, \tilde Z^{(\pi),p}, M^{(\pi,p})_{p\in\mbb N\cup\{0\}}$.

In view of the decomposition
$$Y^{(\pi)}- Y = Y^{(\pi)}-Y^{(\pi),p} + Y^{(\pi),p}-Y^{\infty,p} + Y^{\infty,p}-Y$$
and the fact that $Y^{\infty,p}$ converges to $Y$ and $Y^{(\pi),p}$ to $Y^{(\pi)}$  in $\mathcal S^2$-norm
as $p\to\infty$ (see Tang \& Li (1994) and \eqref{cor:picbsdde} above, respectively), we have that the convergence of
$Y^{(\pi)}$ to $Y$ in the Skorokhod metric in $L^2$ will follow once we show
that  $Y^{(\pi),p}$ converges to $Y^{\infty,p}$ in the latter sense, for any fixed $p$:
\begin{lemma}\label{lem:convpn}
Let $p\in\mbb N$. Then we have
\begin{equation}
\E[d_S^2(Y^{(\pi),p}, Y^{\infty,p})] \to 0, \q\text{as $\Delta\to 0$}.
\end{equation}
\end{lemma}
To establish Lemma~\ref{lem:convpn} we first provide a proof in the case of `smooth' drivers and terminal
conditions (in Section~\ref{ssecsmooth}), and use subsequently density arguments to show that the convergence
carries over to the general case (in Section~\ref{ssecdens}).
\end{proof}

\subsection{The smooth case}
\label{ssecsmooth}

In order to show convergence we first restrict to the case that the terminal conditions and driver functions are bounded infinitely (Fr\'{e}chet-)differentiable functionals,
in the following sense
\begin{definition}\rm
 Let $\mathcal{H}$ be a Hilbert space. (i) A function $f:\mathcal{H}\to\mathbb{R}$ is differentiable if it is
 Fr\'echet-differentiable in every $l\in\mathcal{H}$, {\em i.e.}, there exists a bounded linear operator
 $A_l:\mathcal{H}\to\mathbb{R}$ satisfying
  $$ \lim_{h\to 0} \frac{f(h+l)-f(l)-A_l(h)}{|h|}=0  .$$
We set $D^{(1)}f(l)=A_l.$

(ii) A function $f:\mathcal{H}\to\mathbb{R}$ is
 $k$-times differentiable in $l$, $k\in\mbb N$, if there exists a bounded $k$-linear map $A_l:\mathcal{H}^k\to \mathbb{R}$
 such that for every
 $h_1,\ldots,h_{k-1}\in\mathcal{H}$
 $$ \lim_{h_k\to 0} \frac{D^{(k-1)}f(h_{k}+l)(h_1\ldots,h_{k-1})-D^{(k-1)}f(l)(h_1,\ldots,h_{k-1})-A_l(h_1,\ldots,h_{k})}{|h_k|}=0.$$

(iii) A function $f:\mathcal{H}\to\mathbb{R}$ is element of
$C_b^{\infty}(\mathcal{H})$ if all its higher derivatives are
bounded, {\em i.e.}, for every $k\in\mbb N$ there exists
$\bar{C}_k>0$ such that for all $h_i\in\mathcal{H}$
$$\sup_{l\in \mathcal{H}}|D^{(k)}H(l)(h_1,\ldots,h_k)|\leq \bar{C}_k \Prod_{i=1}^k|h_i|.$$
\end{definition}

Given these definitions the formulation of the smoothness condition that is in force throughout this subsection is as follows:
\begin{As}\label{as:smooth}\rm
(i) For some $k\in\mbb N$ and $H\in C_b^\infty(\mathbb R^{2k})$
the terminal conditions $F$ and $F^{(\pi)}$ are given by
\begin{equation*}
\begin{array}{ll}
F^{(\pi)} = H(W^{(\pi)}_{s_1},\ldots, W^{(\pi)}_{s_k}, {X}^{(\pi)}_{s_1},\ldots,{X}^{(\pi)}_{s_k}),\\
F = H(W_{s_1},\ldots, W_{s_k}, {X}_{s_1},\ldots,{X}_{s_k}),
\end{array}
\q \text{for some $s_1, \ldots, s_k\in[0,T]$}.
\end{equation*}
Moreover, $F^{(\pi)}$ converges to $F$ as $\Delta\to 0$ in $L^2(\mathbb P)$.

(ii) The drivers $f$ and $f^{(\pi)}$ satisfy $f(t,\cdot)\in
C_b^{\infty}( \mathbb{R}\times \mathbb{R}^{d_1}\times
L^2(\nu(\td x),\mathcal{B}(\mathbb{R}^{d_2}\setminus\{0\})))$ and
 $f^{(\pi)}(t,\cdot)\in
C_b^{\infty}(\mathbb{R}\times \mathbb{R}^{d_1}\times
L^2(\nu^{(\pi)}(\td x),\mathcal{B}(\mathbb{R}^{d_2}\setminus\{0\})))$
where, for each $k$, the $k$-th derivative of $f^{(\pi)}$ is bounded
uniformly in $t$ and $\Delta$, the mesh of $\pi$.
\end{As}

Under the smoothness conditions given in Assumption~\ref{as:smooth} the corresponding
Picard sequences obey a number of properties that play an important role
in the proof of Proposition~\ref{lem:convpn2}:

\begin{lemma}
\label{lipschitz} \label{tildez}
{\rm (i)} Let $p\in\mathbb{N}$. There exists a
constant $\bar{K}_p>0$ satisfying for all partitions $\pi$ \be \label{lipz}
|\tilde{Z}^{(\pi),p}_s(x)|\leq \bar{K}_p |x|\mbox{ for all
}x\in\mathbb{R}^{d_2}, \,\, s\in [0,T],\ee
where $\tilde{Z}^{(\pi),p}_s$ denotes a continuous version (in $x$).
Furthermore, $Y^{(\pi),p}$
and $Z^{(\pi),p}$ are uniformly bounded over partitions $\pi$.

{\rm (ii)} $\tilde{Z}^{\infty,p}$ are uniformly
Lipschitz-continuous in $x$, {\em i.e.},
there exists a constant $K'_p>0$
such that $|\tilde{Z}^{\infty,p}_t(x)| \leq K'_p |x|  $ for all
$x\in\mathbb{R}^{d_2}$ and every $t\in[0,T]$, where $\tilde{Z}^{\infty,p}_t$ denotes again a continuous version (in $x$).
\end{lemma}

Given these properties, which  proof is given in the Appendix, we show the convergence of
$Y^{(\pi)}$ as stated in Lemma~\ref{lem:convpn} and in addition
the convergence in mean-square of the triplet
$(Z^{(\pi)},\tilde Z^{(\pi)},M^{(\pi)})$ to $(Z,\tilde Z, 0)$:
\begin{prop}\label{lem:convpn2}
For any $p\in\mbb N$ we have as $\Delta\searrow0$
\begin{eqnarray}
\label{1}&& \E[d_S^2(Y^{(\pi),p}, Y^{\infty,p})] \to 0,\\
\label{2}&& \E\le[\int_0^T\le\{ |
Z^{(\pi),p}_s - Z^{\infty,p}_s |^2 +
\int_{
\mathbb{R}^{d_2}\setminus\{0\}}| \tilde{Z}^{(\pi),p}_s(x) -
\tilde{Z}^{\infty,p}_s(x) |^2\nu(\td x)\ri\} \td s + |M^{(\pi),p}_T|^2\ri] \to 0.
\end{eqnarray}
\end{prop}
\begin{proof} The proof is based on an induction with respect to $p$.
We note that the assertions are trivially satisfied for $p=0$. Assuming that the assertion is
satisfied for a certain $p$ we show next that \eqref{1} and \eqref{2} are satisfied for $p+1$.

{\em Proof of ~\eqref{1} with $p$ replaced by $p+1$:}
In view of the uniform Lipschitz continuity of the driver functions $f^{(\pi)}$
and since these are piecewise constant
we have
\begin{eqnarray}
\nonumber
\lefteqn{ \limsup_{\Delta\to 0}\sup_{t\in[0,T]}
\Big| \sum_{j: t_j\in\pi\cap[0,t]}
f^{(\pi)}(t_j,Y^{(\pi),p}_{t_j},Z^{(\pi),p}_{t_j},\tilde{Z}^{(\pi),p}_{t_j})\Delta -
\int_0^t
f(s,Y^{\infty,p}_{s},Z^{\infty,p}_s,\tilde{Z}^{\infty,p}_s) \td s
\Big|}\\ \nonumber
&\leq& \limsup_{\Delta\to 0} \int_{0}^T
|f^{(\pi)}(s,Y^{(\pi),p}_{s},Z^{(\pi),p}_s,\tilde{Z}^{(\pi),p}_s)
-f^{(\pi)}(s,Y^{\infty,p}_{s},Z^{\infty,p}_s,\tilde{Z}^{\infty,p}_s)|\td s\\ \nonumber
&\phantom{=}& + \limsup_{\Delta\to 0}\int_0^T
|f^{(\pi)}(s,Y^{\infty,p}_{s},Z^{\infty,p}_s,\tilde{Z}^{\infty,p}_s)-
f(s,Y^{\infty,p}_{s},Z^{\infty,p}_s,\tilde{Z}^{\infty,p}_s)|\td s\\
&\leq&
\limsup_{\Delta\to 0}
K\Big(\int_0^T\Big[|Y^{(\pi),p}_s-Y^{\infty,p}_s|+
|Z^{(\pi),p}_s-Z^{\infty,p}_s|+\sqrt{\E_{\nu^{(\pi)}}([\tilde{Z}^{(\pi),p}_s(\xi)
-\tilde{Z}^{\infty,p}_s(\xi)]^2)}\Big]\td s\Big),  \label{estK}
\end{eqnarray}
where in the third line the limsup vanishes in view of Assumption~\ref{A2} and Lemma~\ref{lipschitz}(ii).
Using Lemmas \ref{sw} and~\ref{lipschitz} we find for any $s\in[0,T]$
\begin{eqnarray}
\lim_{\Delta\to 0} \E_{\nu^{(\pi)}}([\tilde{Z}^{(\pi),p}_s(\xi)
-\tilde{Z}^{\infty,p}_s(\xi)]^2)
=  \lim_{\Delta\to 0} \int_{\mathbb{R}^{d_2}\setminus\{0\}}|\tilde{Z}^{(\pi),p}_s(x)
-\tilde{Z}^{\infty,p}_s(x)|^2\nu(\td x). \label{estK2}
\end{eqnarray}
The induction assumption implies that the right-hand sides of \eqref{estK} and \eqref{estK2}
are equal to zero, where the limits are in $L^2$.
By combining the convergence in $\mathcal H^2$-norm of the drivers and the extended convergence in
Proposition~\ref{thm:ewx} (see also the remark after Definition~\ref{def:extco})
 we find that as $\Delta\searrow 0$
\begin{multline*}
m^{(\pi)}_t:=\E\le[F^{(\pi)}+\sum_{j: t_j\in\pi}
f^{(\pi)}(t_j,Y^{(\pi),p}_{t_j}, Z^{(\pi),p}_{t_j},\tilde{Z}^{(\pi),p}_{t_j})\Delta
\Bigg|\F^{(\pi)}_t\ri] \\  \q\q\longrightarrow m_t:= \E\le[F+\int_{0}^T f(s,Y^{\infty,p}_{s},
Z^{\infty,p}_s,\tilde{Z}^{\infty,p}_s) \td s \Bigg|\F_t\ri],
\end{multline*}
and as a consequence also
\begin{multline*}
Y^{(\pi),p+1}_t = m^{(\pi)}_t - \sum_{j: t_j\in\pi\cap[0,t]}
f^{(\pi)}(t_j,Y^{(\pi),p}_{t_j}, Z^{(\pi),p}_{t_j},\tilde{Z}^{(\pi),p}_{t_j})\Delta
\to Y^{\infty,p+1}_t = m_t - \int_0^tf(s,Y^{\infty,p}_s, Z^{\infty,p}_s, \tilde Z^{\infty,p}_s)\td s,
\end{multline*}
where the convergence is in probability in Skorokhod $J_1$-topology.

As $Y^{(\pi),p+1}$ is uniformly bounded over partitions $\pi$ (Lemma~\ref{lipschitz}(i)),
we deduce that $\E[d^2_S(Y^{(\pi),p+1}, Y^{\infty,p+1})]$ tends to zero as $\Delta\to 0$,
so that \eqref{1} holds with $p$ replaced by $p+1$.

{\em Proof of ~\eqref{2} with $p$ replaced by $p+1$:} The argument consists of a number of steps that are listed
in the following auxiliary result:
\begin{lemma}\label{lem:aux} The following convergence holds in the supremum norm in probability as $\Delta\to 0$:
\begin{eqnarray}\label{2one}
&&\q \int_{0}^{\cdot}|Z^{(\pi),p+1}_s|^2 \td s +
\int_{[0,\cdot]\times \mathbb{R}^{d_2}\setminus\{0\}}
|\tilde{Z}^{(\pi),p+1}_s(x)|^2
\nu(\td x)\td s+\<M^{(\pi),p+1}\>_{\cdot}\\ \nonumber
&& \quad\quad\quad \longrightarrow \int_{0}^{\cdot} |Z^{\infty,p+1}_s|^2\td s +
\int_{[0,\cdot]\times \mathbb{R}^{d_2}\setminus\{0\}}
|\tilde{Z}^{\infty,p+1}_s(x)|^2 \nu(\td x)\td s,
\\
&&\q \int_{0}^{\cdot} Z^{(\pi),p+1}_s
\td s \longrightarrow \int_{0}^{\cdot} Z^{\infty,p+1}_s \td s,
\label{2post2}\\
\label{2temp1}
&&\q \int_{[0,\cdot]\times \mathbb{R}^{d_2}\setminus\{0\}}
\tilde{Z}^{(\pi),p+1}_s(x)\bar{Z}_s(x) \nu(\td x)\td s \longrightarrow
\int_{[0,\cdot]\times
\mathbb{R}^{d_2}\setminus\{0\}}\tilde{Z}^{\infty,p+1}_s(x)\bar{Z}_s(x)\nu(\td x)\td s,
\end{eqnarray}
for any function
$\bar{Z}:[0,T]\times\mathbb{R}^{d_2}\setminus\{0\}\to\mathbb{R}$
that is bounded, jointly continuous, and zero in an
environment around zero. Furthermore, we have the following convergence in $L^1$:
 \begin{eqnarray}
&&\q \int_0^T\le\{|Z^{(\pi),p+1}_s -Z^{\infty,p+1}_s |^2+\int_{\mathbb{R}^{d_2}\setminus\{0\}}
|\tilde{Z}^{(\pi),p+1}_s(x)-\tilde{Z}^{\infty,p+1}_s(\omega,x)|^2 \nu(\td x)\ri\}\td s + |M^{(\pi),p+1}_T|^2 \longrightarrow 0, \label{3temp}
\end{eqnarray}
\end{lemma}
It follows from \eqref{3temp} that \eqref{2} is valid with $p$ replaced by $p+1$, and thus
the proof of the proposition is complete.
\end{proof}
\begin{proof}[Proof of Lemma~\ref{lem:aux}:] The proof is given in four parts (corresponding to the different equations):

{\it Proof of ~\eqref{2one}:} We show that the assertion follows from
the convergence of the compensators of the martingales $L^{(\pi)}$, defined by
\begin{equation}\label{eq:Lpi}
L^{(\pi)}_t = m^{(\pi)}_t - Y_0^{(\pi),p+1},
\end{equation}
 to
the compensator of the martingale $L=\{L_t = m_t-Y_0^{\infty,p+1}\}$,
by verifying that the conditions of Theorem~\ref{coro12} are satisfied.
We first show
\begin{equation}\label{eq:LL2}
\E[d^2_S(L^{(\pi)}, L)]\to 0,\q\text{as $\Delta\to0$.}
\end{equation}
Since the
processes $L^{(\pi)}$ converge to $L$ in probability in the Skorokhod $J_1$-topology
(see the end of the proof of ~\eqref{1}),
the convergence in \eqref{eq:LL2} follows by the lemma de la Vall\'{e}e-Poussin
and the fact that the collection $(L^{(\pi)})_{\pi}$ is uniformly bounded, as
\begin{eqnarray}
\label{dela}
\sup_\pi \|L_T^{(\pi)}\|_\infty &\leq& \sup_{\pi}  \|F^{(\pi)}\|_\infty+\sup_{\pi} |Y^{(\pi)}_0|+T\sup_{\pi, t}|f^{(\pi)}(t,0,0,0)|+ B^{(\pi)}\\ \nonumber
&\phantom{=}& + KT\Big(\sup_{\pi}
\left\|\sup_t|Y^{(\pi),p}_t|\right\|_\infty+\sup_\pi
\left\|\sup_t|Z^{(\pi),p}_t|\right\|_\infty\Big),
 \end{eqnarray}
where $||\cdot||_\infty$ denotes the essential-supremum norm
and where by Jensen's inequality, the independence of increments and the conditions \eqref{XX1} and~\eqref{X2}, we have
\begin{eqnarray*}
B^{(\pi)} &:=&  \sum_i \Delta
K\bar{K}_{p+1}\sqrt{ \int_{\mbb R^{d_2}\backslash\{0\}} |x|^2 \nu^{(\pi)}(\td x)}
\leq \sqrt{T}K \bar{K}_{p+1}  \sqrt{\sum_i
 \int_{\mbb R^{d_2}\backslash\{0\}} |x|^2 \nu^{(\pi)}(\td x)\Delta}
 \\
&=& \sqrt{T}K\bar{K}_{p+1} \sqrt{\E[|X_T^{(\pi)}|^2]}
\to \sqrt{T}K\bar{K}_{p+1} \sqrt{\E[|X_T|^2]}.
\nonumber
 \end{eqnarray*}

Since we have the convergence of $L^{(\pi)}_T$ to $L_T$ in $L^2$ (from \eqref{eq:LL2})
and the extended convergence $(L^{(\pi)},\mc F^{(\pi)})\to (L,\mc F)$
(from \eqref{eq:LL2}, Proposition~\ref{thm:ewx} and Remark \ref{martconv}) it follows from Theorem \ref{coro12} that $(\<L^{(\pi)},L^{(\pi)}\>_t)$
converges to $(\<L,L\>_t)$ in probability in the Skorokhod $J_1$-topology.
By the orthogonality of the martingales $M^{(\pi)}$, $W^{(\pi)}$ and the point process induced by $X^{(\pi)}$
on the one hand and the orthogonality 
of $W$ and $\tilde N$ on the other hand we find
\begin{eqnarray}\nonumber
&&\!\!\!\!\!\!\!\!\sum_{t_i\in\pi\backslash\{T\}\cap [0,\cdot]}\le\{|Z^{(\pi),p+1}_{t_i}|^2\Delta
+ \E_{t_i}[|\tilde{Z}^{(\pi),p+1}_{t_i}(\Delta X^{(\pi)}_{t_i})|^2]\ri\}
- \sum_{t_i\in\pi\backslash\{T\}\cap [0,\cdot]} \Big|\E_{t_i}\le[\tilde{Z}^{(\pi),p+1}_{t_i}(\Delta X^{(\pi)}_{t_i})\ri]\Big|^2
+\<M^{(\pi),p+1}\>_{\cdot}\\&& \q\q\q\q\q\q
\longrightarrow\int_{0}^{\cdot} |Z^{\infty,p+1}_s|^2 \td s +\int_{[0,\cdot]\times \mathbb{R}^{d_2}\setminus\{0\}}
|\tilde{Z}^{\infty,p+1}_s(x)|^2 \nu(\td x)\td s \label{eq:compc}
\end{eqnarray}
in the supremum norm in
probability. In this display we note that the second sum vanishes as $\Delta$ tends to zero.
More specifically, in view of Lemma \ref{lipschitz} and the condition in ~\eqref{XX0} we have
\begin{equation}\label{eq:sumc}
\sum_{t_i\in\pi\backslash\{T\}}
 \Big|\E_{t_i}\le[\tilde{Z}^{(\pi),p+1}_{t_i}(\Delta X^{(\pi)}_{t_i})\ri]\Big|^2
\leq \bar{K}_{p+1}^2
\sum_{t_i\in\pi\backslash\{T\}}
\le|\E_{t_i}\le[|\Delta X^{(\pi)}_{t_i}|\ri]\ri|^2 \to 0
\end{equation}
as $\Delta$ tends to zero, where we used that $\E_{t_i}[|\Delta X^{(\pi)}_{t_i}|] = \E[|\Delta X^{(\pi)}_{t_i}|]$ by the independence of the increments of $X^{(\pi)}$.
The assertion in ~\eqref{2one} follows by combining~\eqref{eq:compc} and \eqref{eq:sumc}
with the fact that $Z^{(\pi),p+1}_s$ is piecewise constant as function of  $s$ and with
Lemma~\ref{sw}(ii), which is applicable as
$(\tilde Z^{(\pi),p+1})_\pi$ is uniformly Lipschitz-continuous.

{\it Proof of ~\eqref{2post2}:}
It follows from Corollary~\ref{cor:cv}(i) and the representation \eqref{fwrep} of the square-integrable martingale $L^{(\pi)}$ defined in \eqref{eq:Lpi}
that as $\Delta\to0$
\begin{eqnarray*}
\< W^{(\pi)}, L^{(\pi)}\>_{\cdot} = \sum_{i=0}^{\lfloor \cdot N\rfloor-1}|Z^{(\pi),p+1}_{t_i}|\Delta
\to  \<W,L\>_{\cdot}=\int_{0}^{\cdot} Z^{\infty,p+1}_s \td s,
\end{eqnarray*}
in the supremeum norm in probability which implies the assertion in ~\eqref{2post2}.

{\it Proof of ~\eqref{2temp1}:}
We conclude from Corollary~\ref{cor:cv}(ii) and the representation of the martingale $L^{(\pi)}$
\begin{align*}
& \lim_{\Delta\to 0}\sum_{t_i\in\pi\backslash\{T\}\cap [0,\cdot]}\left\{ \E_{t_i}[\tilde{Z}^{(\pi),p+1}_{t_i}(\Delta X^{(\pi)}_{t_i})\bar{Z}_{t_i}(\Delta X^{(\pi)}_{t_i})]
- \E_{t_i}[\tilde{Z}^{(\pi),p+1}_{t_i}(\Delta X^{(\pi)}_{t_i})] \E_{t_i}[\bar{Z}_{t_i}(\Delta X^{(\pi)}_{t_i})]\ri\}\\
&\q\q\q \to \int_{[0,\cdot]\times \mathbb{R}^{d_2}\setminus\{0\}}
\tilde{Z}^{\infty,p+1}_s(x)\bar{Z}_s(x) \nu(\td x)\td s,\end{align*}
in probability in the Skorokhod $J_1$-topology. As the limit is continuous, this convergence
also holds in the supremum norm. Moreover, as $\bar{Z}$ is bounded,
continuous, and zero in an environment around zero, it is clear that
there exists $\hat{K}>0$ such that $|\bar{Z}_s(x)|\leq \hat{K}
|x|$. It follows then from Lemma~\ref{sw}(ii) that we have the convergence in
\eqref{2temp1} in the supremum norm in probability.

{\em Proof of ~\eqref{3temp}:}
Next let us switch to a
subsequence and assume that all previous convergence results derived in the proofs of \eqref{2one}--\eqref{2temp1}
hold for a.e.~$\omega\in\Omega$. Fix such an $\omega\in\Omega$. By Lemma \ref{lipschitz} there
exists constants $\bar{K}_{p+1}>0$ such that
$$\sup_{\pi} \int_{[0,T]\times \mathbb{R}^{d_2}\setminus\{0\}}\Big|\tilde{Z}^{(\pi),p+1}_s(\omega,x)\Big|^2\nu(\td x)\td s\leq
 \bar{K}^2_{p+1}\int_{[0,T]\times \mathbb{R}^{d_2}\setminus\{0\}}|x|^2\nu(\td x)\td s=
T\bar{K}^2_{p+1}\int_{\mathbb{R}^{d_2}\setminus\{0\}}|x|^2\nu(\td x)<\infty. $$
Hence,
$\tilde{Z}^{(\pi),p+1}_s(\omega,x)$ is
 uniformly
bounded in $L^{2}(\nu(\td x)\times \td s)$. By switching to a
subsequence, we may assume that $\tilde{Z}^{(\pi),p+1}_s(\omega,x)$
converges weakly in $L^2(\nu(\td x)\times
\td s,\mathcal{B}(\mathbb{R}^{d_2}\setminus\{0\})\otimes
\mathcal{B}([0,T]))$ to a limiting function. From (\ref{2temp1})
it follows that this limit is equal to $\tilde{Z}^{\infty,p+1}_\cdot(\omega,\cdot)$.
Furthermore, by (\ref{2post2}) we also know that for a.e.~$\omega$ we have that $Z^{(\pi),p+1}_{\cdot}(\omega)$ converges
weakly to $Z^{\infty,p+1}_{\cdot}(\omega)$ in $L^2_{d_1}(\td s).$
We also have that the pairs
$(Z^{n,p+1}(\omega),\tilde{Z}^{n,p+1}(\omega))$ converge
weakly $(Z^{\infty,p+1}(\omega),\tilde{Z}^{\infty,p+1}(\omega))$
in $L^2_{d_1}(ds)\times
L^2(\nu(dx)\times ds)$
equipped with the inner product
$$\<(z^1,\tilde{z}^1),(z^2,\tilde{z}^2)\>_*=\int_0^T z^1_s z^2_s \td s+\int_{[0,T]\times \mathbb{R}^{d_2}\setminus\{0\}}\tilde{z}^1_s(x)
\tilde{z}^2_s(x) \nu(\td x)\td s.$$
Denoting by $\|\cdot\|_*$ the norm
associated to this inner product
we have by (\ref{2one})
\begin{align*} \limsup_{\Delta\to 0}
\|I_\Delta\|^2_*&:=\limsup_{\Delta\to 0}\left\|(Z^{(\pi),p+1}(\omega),\tilde{Z}^{(\pi),p+1}(\omega))\right\|^2_*
\leq
\left\|(Z^{\infty,p+1}(\omega),\tilde{Z}^{\infty,p+1}(\omega))\right\|_*^2=:||I||_*^2.
\end{align*}
Therefore, we have
$$0\leq \limsup_{\Delta\to 0}\<I_\Delta-I,I_\Delta-I\>=\limsup_{\Delta\to 0} (\<I_\Delta,I_\Delta\>-2\<I,I_\Delta\>+\<I,I\>)\leq
\<I,I\>-2\<I,I\>+\<I,I\>=0 .$$
Hence, all inequalities must be equalities and we get that
$$\int_0^T|Z^{(\pi),p+1}_s(\omega) -Z^{\infty,p+1}_s(\omega) |^2\td s+
\int_{[0,T]\times \mathbb{R}^{d_2}\setminus\{0\}}
|\tilde{Z}^{(\pi),p+1}_s(\omega ,x)-\tilde{Z}^{\infty,p+1}_s(\omega,x)|^2 \nu(\td x)\td s\to 0
\mbox{ as $\Delta\to0$}.$$
By (\ref{2one}) it follows that also $|M^{(\pi),p+1}_t(\omega)|^2$ converges to zero as well.
Therefore, for a.e.~$\omega$, for any subsequence on the
left-hand side in (\ref{3temp}) we can find a subsubsequence
converging to zero. Thus, we must have convergence in
probability in (\ref{3temp}). We note that, for any $p$, $(M_T^{(\pi),p})^2$
is uniformly integrable over partitions $\pi$, since we have the bound
$$
\sup_\pi \E[\langle M^{(\pi),p}, M^{(\pi),p}\rangle_T^2]
\leq \sup_\pi \E[\langle L^{(\pi),p}, L^{(\pi),p}\rangle_T^2]
\leq \sup_\pi \bar{C}\|L^{(\pi),p}_T\|^4_\infty < \infty, \mbox{ for a }\bar{C}>0,
$$
which follow by the definitions of $M^{(\pi),p}$ and $L^{(\pi),p}$,
the BDG and Doob inequalities, and the fact that
$L_T^{(\pi)}$ is bounded uniformly in $\pi$ (by ~\eqref{dela}).

That the convergence in ~\eqref{3temp} also holds true in $L^1$ may be seen from another application of the Lemma de la Vall\'{e}e-Poussin, which is applicable as
the integral on the left-hand side is bounded uniformly in $\pi$ (by Lemma \ref{lipschitz}),
in combination with the uniform integrability of $(M_T^{(\pi),p+1})^2$.
\end{proof}
\subsection{Density argument}\label{ssecdens}
We complete the proof of Theorem~\ref{thm} by combining
Proposition~\ref{lem:convpn2} with a density argument.

\begin{proof}[Proof of Theorem~\ref{thm}]
Let $k\in\mathbb N$ be arbitrary. By standard density results
we can find functions $H_{k}\in C_b^\infty( \mathbb{R}^{2k})$ and uniformly
$K$-Lipschitz-continuous functions $f^k$ and $f^{(\pi),k}$ such
that $f^{k}(t,\cdot)\in C_b^{0,\infty}( \mathbb{R}\times
\mathbb{R}^{d_1}\times L^2(\nu(\td x))$, and $f^{(\pi),k}(t,\cdot)\in
C_b^{0,\infty}( \mathbb{R}\times \mathbb{R}^{d_1}\times
L^2(\nu^{(\pi)}(\td x))$ converging to $f^k$ with
\begin{eqnarray}\nonumber
\lefteqn{T \sup_{t,y,z,\tilde{z}}(|f(t,y,z,\tilde{z})-f^k(t,y,z,\tilde{z})| + |f^{(\pi),k}(t,y,z,\tilde{z})-f^{(\pi)}(t,y,z,\tilde{z})|)
} \\ &+& \E[|F-H_{k}(W_{s_1},{X}_{s_1},\ldots,W_{s_k},{X}_{s_k})|^2] \leq \frac{1}{k}.
\label{bbb}
\end{eqnarray}
The triangle inequality for the Skorokhod metric $d_S$ and the inequality $(x+y+z)^2\leq 3(x^2+y^2+z^2)$ imply
\begin{eqnarray}\nonumber
\lefteqn{\E[d^2_S(Y^{(\pi)},Y)]}\\
&\leq& 3\E[d^2_S(Y^{(\pi)},\tilde Y^{(\pi)})]
+ 3\E[d^2_S(\tilde Y^{(\pi)},\tilde Y)] +
3\E[d^2_S(\tilde Y,Y)] := 3d^2_1(k,\pi) + 3d^2_2(k,\pi) + 3d^2_3(k),
\label{tri}
\end{eqnarray}
where $\tilde Y$ and $\tilde Y^{(\pi)}$ denote the solutions of the BSDE and BS$\Delta$E
with terminal conditions $\tilde F=H_k(W,X)$, $\tilde F^{(\pi)} = H_k(W^{(\pi)}, X^{(\pi)})$
and drivers $\tilde f = f^k$ and $\tilde f^{(\pi)} = f^{(\pi),k}$, respectively.

We first estimate the distances between $Y^{(\pi)}$ and $\tilde Y^{(\pi)}$ and between
$Y$ and $\tilde Y$  in the supremum norm.
By applying Theorem \ref{thmstable} and Remark~\ref{remstable} we see that
the following bounds hold true:
\begin{eqnarray*}
&& \mbf d_1^2(k,\pi):=\E\le[\sup_{t\in[0,T]}|Y^{(\pi)}_t - \tilde Y^{(\pi)}_t|^2\ri] \leq
\bar C\E\le[|F^{(\pi)}-\tilde F^{(\pi)}|^2+\int_0^{T} |\delta f^{(\pi),k}(s,Z^{(\pi)}_s,\tilde{Z}^{(\pi)}_s)|^2\td s\right],\\
&& \mbf d^2_3(k):=\E\le[\sup_{t\in[0,T]}| Y_t-\tilde Y_t |^2\ri] \leq \bar{c}\,\E\left[|F-\tilde F|^2 + \int_0^{T} |\delta f^{k}(s,Z_s,\tilde{Z}_s)|^2\td s\right],
\end{eqnarray*}
where we denote $\delta f^k:=f-f^k$ and $\delta f^{(\pi),k}:=f^{(\pi)}-f^{(\pi),k}$.
By deploying the bound \eqref{bbb} and Proposition \ref{lem:convpn2}
 and using that $F^{(\pi)} \in L^2(\F^{(\pi)})$ converges to $F\in L^2(\F) $ in $L^2$,
we find
\begin{eqnarray}\label{eq:lims}
&& \limsup_{\Delta\to 0}\mbf d^2_1(k,\pi) \leq  \bar C\E[|F-\tilde F|^2]+\frac{\bar{C}}{k} \leq \frac{2\bar{C}}{k},
\quad
\limsup_{\Delta\to 0} d^2_2(k,\pi) = 0,
\quad
\mbf d^2_3(k) \leq \frac{2\bar c}{k}.
\end{eqnarray}
Since the Skorokhod metric is dominated by the supremum norm (see {\em e.g.} Eqn. VI.1.26 in
Jacod \& Shiryaev (2003)) we conclude from \eqref{tri}
and \eqref{eq:lims} that $\limsup_{\Delta\to 0}\E[d^2_S(Y^{(\pi)},Y)] \leq 6 (\bar C+\bar c )/k$
for arbitrary $k$. Hence the proof is complete.
\end{proof}
\section{Example}\label{sec:example}
By way of illustration we specify in this section a sequence of approximating BS$\Delta$Es
driven by a discrete-valued approximating sequence $(X^{(\pi)})_\pi$.
We consider the BSDE
\begin{equation}\label{bsde.ex}
Y_t = F + \int_t^T f(s, Y_s, \tilde Z_s)\td s - \int_{(t,T]\times \mbb R\backslash\{0\}}\tilde Z_s(x)
\tilde N(\td s\times \td x), \q t\in[0,T],
\end{equation}
which is driven by the compensated Poisson random measure $\tilde N$
associated to a square-integrable zero-mean
real-valued L\'{e}vy process $X$ ($d_2=1$). Here,
$f:[0,T]\times \mbb R\times L^2(\nu(\td x), \mc B(\mbb R\backslash\{0\}))\to\mbb R$ is
the driver function. As usual we assume that $f$ is continuous as function of $t$,
and uniformly Lipschitz continuous (as in \eqref{f} without the Brownian term).
We consider final conditions $F$ of the form
\begin{equation}
F = H(X_{s_0}, \ldots, X_{s_D})\q\text{with $s_i - s_{i-1} = \Delta_0$, $s_0=0$, $s_D=T$},
\end{equation}
for some Lipschitz function $H:\mbb R^{D+1}\to\mbb R$ (with Lipschitz constant $K$ say).
We also suppose that the L\'{e}vy measure $\nu$ of $X$
has Blumenthal-Getoor index\footnote{The Blumenthal-Getoor index $\beta$
of $X$ is $\beta=\inf\{p>0: \int_{\{|x|<1\}}|x|^p\nu(\td x) <\infty\}$.} $\beta<2$
and admits a strictly positive
density $g_{\nu}$ on $\mbb R\backslash\{0\}$ satisfying the integrability condition
\begin{equation}\label{gn}
\int_{\{|x|>1\}} g_\nu(x)|x|^{2+\e}\td x < \infty, \q \text{for some $\e>0$}.
\end{equation}
We mention that, under the integrability condition \eqref{gn}, $\E[|X_t|^{2+\e}]$ is finite for any $t$
(see Sato~(1999)), so that in particular $F$ is square-integrable.

For the ease of presentation we consider BS$\Delta$Es defined on grids that are refinements of
$\pi_0=\{s_0, s_1, \ldots, s_D\}$.
We next specify the final value $F^{(\pi)}$, the driver $f^{(\pi)}$ and the random walk $X^{(\pi)}$ and
denote the corresponding BS$\Delta$E on the uniform time-grid $\pi\subset\pi_0$ by
\begin{equation}\label{eq:Ypi}
Y^{(\pi)}_{t_i} = F^{(\pi)} + \sum_{t_j:t_i\leq t_j<T} f^{(\pi)}(t_j, Y^{(\pi)}_{t_j}, \tilde Z^{(\pi)}_{t_j})\Delta
- \sum_{t_j: t_i\leq t_j<T}
\left\{\tilde Z^{(\pi)}_{t_j}(\Delta X^{(\pi)}_{t_j}) - \E_{t_j}[\tilde Z^{(\pi)}_{t_j}(\Delta X^{(\pi)}_{t_j})]\right\}.
\end{equation}
Define the spatial mesh size $h$ by $h^2 = 3\Delta \Sigma^2$, where
 $\Sigma^2 = \int_{\mbb R\backslash\{0\}} x^2\nu(\td x)$ and, as before,
 $\Delta$ denotes the mesh of the partition $\pi$. Then
we have $\nu(\{x:|x|>h\})\,\Delta < 1/3$ as
$$
 \nu(\{x:|x|>h\})\, \Delta = (3\Sigma^{2})^{-1} h^2 \nu(\{x: |x|>h\}) < (3\Sigma^{2})^{-1}
 \int_{\{|x|>h\}} x^2\nu(\td x) < \frac{1}{3}.
$$
We define the distribution of the increments of $X^{(\pi)}$
in terms of the averages of the L\'{e}vy measure $\nu$ over certain sets:
$$
\alpha(A) := \frac{1}{\nu(A)} \int_{A}x\nu(\td x), \q A\in\mc B(\mbb R), \nu(A)>0.
$$
If $\alpha([-h,h]^c)\geq 0$ then set $h_-:=h$ and
$h_+:=\inf\{u\ge h: \alpha([-h,u]^c)=0\}$ and similarly if  $\alpha([-h,h]^c) < 0$
then set $h_+:=h$ and
$h_-:=\inf\{\ell\ge h: \alpha([-\ell,h]^c)=0\}$.

Setting $B_{i+1} := (h_+(i), h_+(i+1)]$ and $B_{-i-1}:=[-h_-(i+1),-h_-(i))$ for $i\in\mbb N$
for some strictly increasing sequences $(h_\pm(i))_{i\in\mbb N}$ with $h_+(1) = h_+$ and $h_-(1) = h_-$ and mesh size going to zero,
we define  for any integer $|i|\ge 2$
$$
\P(X_{t_1}^{(\pi)} = x_{i}) = p_{i}\q \text{with}\q p_{i} = \Delta\,\nu(B_{i}), \q
x_i = \frac{1}{\nu(B_i)}\int_{B_i} x\nu(\td x).
$$
Note that with this choice we have
\begin{eqnarray*}
&& \P\le(X_{t_1}^{(\pi)} \notin [h_-,h_+]\ri) = \sum_{i: |i|\ge 2} p_i = \Delta\,\nu(\{x: x\notin [h_-,h_+]\}), \\
&& \E\le[X_{t_1}^{(\pi)} I_{\{X_{t_1}^{(\pi)}\notin[h_-,h_+]\}}\ri] = \sum_{i: |i|\ge 2} x_i p_i
 = 0 = \alpha([h_-,h_+]^c).
\end{eqnarray*}
The description of the distribution of $X^{(\pi)}_{t_1}$ is completed by setting
$\P(X_{t_1}^{(\pi)} = \pm h) = p_{\pm 1}$ and $P(X_{t_1}^{(\pi)} = 0) = p_0$, where $p_0$ and $p_{\pm1}$ are chosen
so as to satisfy the conditions of unit mass and zero mean and to match the instantaneous variance:
$$
\sum_{i:|i|\ge 0}p_i = 1,\q   \sum_{i: |i|\ge 0} x_i p_i = 0, \q
\sum_{i: |i|\ge 0} p_i(x_i)^2 = \Delta \int_{\mbb R\backslash\{0\}} x^2\nu(\td x),
$$
or equivalently, $p_{-1}+p_0+p_{1} = 1 - \nu(\{x: x\notin[h_-,h_+]\})\,\Delta$, and
\begin{multline*}
(p_{1}-p_{-1})h = 0, \q (p_1 + p_{-1})h^2 = \Delta\, \int_{\{x\in[h_-,h_+]\}} x^2\nu(\td x) +
\Delta\,V(h_-,h_+) \\
\Rightarrow
p_{-1} = p_{1} = \frac{1}{6\Sigma^2} \le\{S(h_-,h_+) + V(h_-,h_+)\ri\} \leq \frac{1}{6},
\q p_0 > \frac{1}{3},
\end{multline*}
with
$$S(h_-,h_+) := \int_{[h_-,h_+]} x^2\nu(\td x)\ \text{ and }
V(h_-, h_+) := \int_{\{x\notin[h_-,h_+]\}} x^2\nu(\td x) - \sum_{i: |i|\ge 2}\frac{1}{\nu(B_i)}
\left\{\int_{B_i}x\nu(\td x)
\right\}^2,
$$
which is non-negative as a consequence of the Cauchy-Schwarz inequality.
In particular, we see that the zero-jump-condition \eqref{eq:zero} is satisfied.
We note that the approximating processes $(X^{(\pi)})_\pi$ also satisfy conditions \eqref{XX0} and \eqref{XX1}, since
by construction
$\E[|X_{t_1}^{(\pi)}|^2] = \Delta \int_{\mbb R\backslash\{0\}}
x^2\nu(\td x)$, while the expectation of $\le|X_{t_1}^{(\pi)}\ri|$ is $o(\sqrt{\Delta})$, since we have
\begin{eqnarray*}
\E\le[\le|X_{t_1}^{(\pi)}\ri|\ri] = (p_1+p_{-1})h + \Delta \int_{[h_-,h_+]^c}|x|\nu(\td x),
\end{eqnarray*}
where $p_1+p_{-1}$ tends to zero when $\Delta\to 0$ and the second term is bounded by
$c\cdot h^{2-\beta/2}$ (which is  $o(\sqrt{\Delta})$ as $\Delta\to 0$ since $\beta<2$ by assumption)
with $c=\int |x|^{1+\beta/2}\nu(\td x)/(3\Sigma^2)$ (which is finite by definition of $\beta$
and $\int x^2\nu(\td x) < \infty \Leftrightarrow \E[X_t^2]<\infty$).

Furthermore, it is easily checked that the sequence $(X^{(\pi)})_\pi$ also satisfies the conditions in \eqref{X2} and
\eqref{XX2}. In particular, $X^{(\pi)}\stackrel{\mc L}{\to} X$ as $\Delta\to 0$,
and on a suitably chosen probability space,  $X^{(\pi)}$ converges to $X$ in probability in the Skorokhod $J_1$-topology,
and  $X^{(\pi)}_T$ converges to $X_T$ in $L^2$.

Next we define $F^{(\pi)} = H(X^{(\pi)}_{s_0}, \ldots, X^{(\pi)}_{s_D})$. By the Lipschitz continuity of $H$
and the convergence of $X^{(\pi)}$ to $X$ in $\mc S^2$ (by Doob's maximal inequality) it follows that also $F^{(\pi)}$ converges to $F$ in $L^2$.

Finally, we specify $f^{(\pi)}$ in terms of $f$ by $f^{(\pi)}(t,y,\tilde z) = f(t,y,Q\tilde z)$
with
$$
(Q\tilde z)(x) =
\begin{cases}
\tilde z(x_i), & x\in B_i, i\neq 1, i\neq 0,\\
\frac{x}{\sqrt{2}h}(\tilde z(-h) + \tilde z(+h)), & x\in[h_-, h_+]\backslash\{0\},\\
0, &  x = 0.
\end{cases}
$$
It is straightforward to verify that the drivers $f^{(\pi)}$ satisfy the required regularity conditions.
In particular,
the uniform Lipschitz-continuity of $f^{(\pi)}$ (as in \eqref{KLip})
can be derived as follows:
for any $y_1,y_0\in\mbb R$, $\tilde z_1,\tilde z_0\in L^2(\nu^{(\pi)}, \mc B(\mbb R\backslash\{0\}))$
the Lipschitz continuity of $f$ implies
\begin{eqnarray*}
&& |f^{(\pi)}(t,y_1,\tilde z_1) - f^{(\pi)}(t,y_0,\tilde z_0)| \leq K\le(|y_1-y_0| + \sqrt{I_Q}\ri), \q\text{with}
\q I_Q := \int_{\mbb R\backslash\{0\}} |Q\tilde z_1(x) - Q\tilde z_0(x)|^2\nu(\td x).
\end{eqnarray*}
Inserting the definitions of $Q\tilde z_0$ and $Q\tilde z_1$ shows
\begin{eqnarray*}
\lefteqn{I_Q = \sum_{i:|i|\ge 2} |\tilde z_1(x_i) - \tilde z_0(x_i)|^2 \nu(B_i)
+\ \frac{\le(\tilde z_1(h) - \tilde z_0(h) + \tilde z_1(-h) - \tilde z_0(-h)\ri)^2}{2}
\frac{S(h_-,h_+)}{ h^2}} \\
&\leq& \sum_{i:|i|\ge 2} \frac{p_i}{\Delta} |\tilde z_1(x_i) - \tilde z_0(x_i)|^2
+ \le(\tilde z_1(h) - \tilde z_0(h)\ri)^2\frac{p_1}{\Delta} + \le(\tilde z_1(-h) - \tilde z_0(-h)\ri)^2
\frac{p_{-1}}{\Delta} \\
 &=&  \int_{\mbb R\backslash\{0\}} |\tilde z_1(x) - \tilde z_0(x)|^2\nu^{(\pi)}(\td x).
\end{eqnarray*}

We next move to the description of the solution of the BS$\Delta$E.
Specifically, we have from Proposition~\ref{lem:BSDEsol} that the solution is given by
\begin{eqnarray}\label{eq:Ypi2}
&& Y_{t_i}^{(\pi)} = v_i\le(X_{t_0}^{(\pi)}, \ldots, X_{t_i}^{(\pi)}\ri), \q i=0, \ldots, N-1,\\
&&\tilde Z_{t_i}^{(\pi)}(x) = w_i\le(X_{t_0}^{(\pi)}, \ldots, X_{t_i}^{(\pi)};x\ri) -
w_i\le(X_{t_0}^{(\pi)}, \ldots, X_{t_i}^{(\pi)}; 0\ri),
\end{eqnarray}
with $Y_{t_N}^{(\pi)}=F^{(\pi)}$, for certain functions $v_i:\mbb R^i\to\mbb R$ and $w_i:\mbb R^{i+1}\to\mbb R$
that are specified recursively as follows:
\begin{eqnarray}\nonumber
w_i(\unl z_{0,i};x) &=& \E[Y_{t_{i+1}}^{(\pi)}|
X_{t_0}^{(\pi)}=z_0, \ldots, X_{t_i}^{(\pi)}=z_i, \Delta X_{t_i}^{(\pi)}=x]
= v_{i+1}(\unl z_{0,i}, z_i+x), \q
x\in E^{(\pi)},\\
v_{i}(\unl z_{0,i}) &=&
f^{(\pi)}(t_i, v_i(\underline z_{0,i}), w_i(\underline z_{0,i};\cdot)- w_i(\underline z_{0,i};0))
+ \sum_{x_j\in E^{(\pi)}} v_{i+1}(\underline z_{0,i}, z_i+x_j) p_j,
\label{vi}
\end{eqnarray}
where $\underline z_{0,i} = (z_0, \ldots, z_i)$ and
$E^{(\pi)}=\{x_i: i\in\mbb Z\}$ denotes the support of the step-size distribution of $X^{(\pi)}$.
While in general  $v_i$ is only implicitly defined by \eqref{vi},
the recursion in \eqref{vi} has an explicit solution when the driver $f(t,y,\tilde z)$ (and thus $f^{(\pi)}(t,y,\tilde z)$) is constant as function of $y$.
In this case, the solution $Y$ of the BSDE is translation invariant
in the sense that $Y({F+a})= a + Y(F)$ for $a\in\mbb R$, where $Y(F)$ denotes the solution of the BSDE with final condition $F$
(see Royer~(2006)).
By Theorem~\ref{thm}, the solution $Y^{(\pi)}$ of the BS$\Delta$E \eqref{eq:Ypi} specified in
\eqref{eq:Ypi2} converges to the solution $Y$ of the BSDE \eqref{bsde.ex}
in $L^2$ in the Skorokhod $J_1$-topology.

\appendix
\section{Proof of Theorem~\ref{thmstable}}
The structure of the proof is inspired by that of
an analogous estimate derived in a Wiener setting in
Proposition 7 in Briand {\em et al.}~(2002).

Assuming without loss of generality $t_i=T$ and that $f^{(\pi)}(t,0,0,0)=0,$
and simplifying notation by dropping in the proof the superscripts $(\pi)$ in the solution $(Y^{(\pi)}, Z^{(\pi)}, \tilde Z^{(\pi)}, M^{(\pi)})$,
 we have for $t_j,t_k\in\pi$ with $t_j< t_k$
\begin{eqnarray}\nonumber
\delta Y_{t_j} &=& \delta Y_{t_k}+ \sum_{r=j}^{k-1}\bigg(\delta f^{(\pi)}(t_r,Y^{0}_{t_r},Z_{t_r}^{0},\tilde{Z}_{t_r}^{0}) +  f^{(\pi),1}(t_r,Y^{0}_{t_r},Z_{t_r}^{0},\tilde{Z}_{t_r}^{0})-f^{(\pi),1}(t_r,Y_{t_r}^{1},Z_{t_r}^{1},\tilde{Z}_{t_r}^{1})\bigg)\Delta\\
&\phantom{=}& - \sum_{r=j}^{k-1}\{\delta Z_{t_r}\Delta W^{(\pi)}_{t_r}+
\delta \tilde{Z}_{t_r}(\Delta X^{(\pi)}_{t_r}) - E_{t_r}[\delta \tilde{Z}_{t_r}(\Delta X^{(\pi)}_{t_r})]\}
- (\delta M_{t_k} - \delta M_{t_j}). \label{eq:ddBsde}
\end{eqnarray}
Since the functions $f^{(\pi)}$ are $K$-Lipschitz and assuming without loss of generality $K>1$, we have
\begin{eqnarray*}
|\delta Y_{t_j}| \leq \E_{t_j}\le[|\delta Y_{t_k}| +
K\sum_{r=j}^{k-1}\bigg\{\le|\delta f^{(\pi)}(t_r,Y^{0}_{t_r},Z_{t_r}^{0},\tilde{Z}_{t_r}^{0})\ri|
+ |\delta Y_{t_r}| + |\delta Z_{t_r}|
+ \sqrt{\E_{\nu^{(\pi)}}([\delta \tilde Z_{t_r}(\xi)]^2)}\bigg\}\Delta\ri]
\end{eqnarray*}
and an application Doob's inequality yields for $t_m< t_k$ with $t_m,t_k\in\pi\backslash\{0\}$
\begin{eqnarray}
\mathbb{E}\left[\sup_{m \leq j < k} |\delta Y_{t_j}|^2\right] &\leq&
4\mathbb{E}\bigg[ \bigg(|\delta Y_{t_k}|
 + K \sum_{r=m}^{k-1}\bigg\{\le|\delta f^{(\pi)}(t_r,Y^{0}_{t_r},Z_{t_r}^{0},\tilde{Z}_{t_r}^{0})\ri|
\nonumber
 \\
&\phantom{=}& + |\delta Y_{t_r}| + |\delta Z_{t_r}|
+ \sqrt{\E_{\nu^{(\pi)}}([\delta \tilde Z_{t_r}(\xi)]^2)}\bigg\}\Delta\bigg)^2\bigg].
\label{ffour}
\end{eqnarray}
Since $W^{(\pi)}$, $X^{(\pi)}$ and $\delta
M^{(\pi)}$ are orthogonal martingales, we have
\begin{eqnarray}\label{eq:est1}
\lefteqn{\E\le[\left|
\sum_{r=m}^{k-1}
\le\{\delta Z_{t_r} \Delta W^{(\pi)}_{t_r} +
\delta\tilde{Z}_{t_r}(\Delta X^{(\pi)}_{t_r})
- \E_{t_r}[\delta\tilde{Z}_{t_r}(\Delta X^{(\pi)}_{t_r})]\ri\}
+ \delta M_{t_{k}} -
\delta M_{t_m}\right|^2\ri]}\\
&=&
\E\le[\sum_{r=m}^{k-1}|\delta Z_{t_r}|^2\Delta +
\<\delta \tilde M\>_{t_{k}} - \<\delta \tilde M\>_{t_{m}} +
\<\delta M\>_{t_{k}} - \<\delta M\>_{t_{m}}
\ri] \nonumber
\end{eqnarray}
where $\delta\tilde M$ is the martingale that is piecewise constant (outside the partition $\pi$)
and has increment $\delta\tilde M_{t_{r+1}} -
\delta\tilde M_{t_{r}}$ given by $\ell_{t_r}(\Delta X^{(\pi)}_{t_r}):=
\delta\tilde{Z}_{t_r}(\Delta X^{(\pi)}_{t_r}) - \E_{t_r}[\delta\tilde{Z}_{t_r}(\Delta X^{(\pi)}_{t_r})]$, and
 $\<\delta \tilde M\>$ and $\<\delta M\>$ denote the predictable compensators
of $\delta M$ and $\delta\tilde M$, which are equal to
$$
\<\delta M\>_{t_{i}} = \sum_{t_j\leq t_{i-1}} \E_{t_j}[|\Delta M^{(\pi)}_{t_j}|^2],\q
 \<\delta \tilde M\>_{t_{i}} = \sum_{t_j\leq t_{i-1}} \E_{t_j}[|\Delta \ell_{t_j}(X^{(\pi)}_{t_j})|^2].
$$
Using the fact that $f^{(\pi),1}$ is $K$-Lipschitz
using ~\eqref{eq:ddBsde} we obtain
\begin{eqnarray}\nonumber
\lefteqn{ \left|
\sum_{r=m}^{k-1}
\le\{\delta Z_{t_r} \Delta W^{(\pi)}_{t_r} +
\delta\tilde{Z}_{t_r}(\Delta X^{(\pi)}_{t_r})
- \E_{t_r}[\delta\tilde{Z}_{t_r}(\Delta X^{(\pi)}_{t_r})]\ri\}
+ \delta M_{t_{k}} -
\delta M_{t_m}
\right|}
\\ \nonumber
&\leq&
|\delta Y_{t_{k}}| + K\sum_{r=m}^{k-1}\bigg\{\le|\delta f^{(\pi)}(t_r,Y^{0}_{t_r},Z_{t_r}^{0},\tilde{Z}_{t_r}^{0})\ri|
+ |\delta Y_{t_r}| \\
&\phantom{=}&
+ |\delta Z_{t_r}|
+ \sqrt{\E_{\nu^{(\pi)}}([\delta \tilde Z_{t_r}(\xi)]^2)}\bigg\}\Delta
+ \sup_{m\leq r < k}{|\delta Y_{t_r}|}.
\label{eq:est2}
\end{eqnarray}
By combining the estimates in~\eqref{ffour}, \eqref{eq:est1}, \eqref{eq:est2} we get
\begin{multline*}
\E\le[ \sup_{m\leq r < k}|\delta Y_{t_r}|^2  +
\sum_{r=m}^{k-1}|\delta Z_{t_r}|^2\Delta +
\<\delta \tilde M\>_{t_{k}} - \<\delta \tilde M\>_{t_m} +
\<\delta M\>_{t_{k}} - \<\delta  M\>_{t_m}
\ri]\\
\leq 14\,
 \mathbb{E}\bigg[\bigg(|\delta Y_{t_k}|
+ K\sum_{r=m}^{k-1}\bigg\{\le|\delta f^{(\pi)}(t_r,Y^{0}_{t_r},Z_{t_r}^{0},\tilde{Z}_{t_r}^{0})\ri|
 \\
 + |\delta Y_{t_r}| + |\delta Z_{t_r}|
+ \sqrt{\E_{\nu^{(\pi)}}([\delta \tilde Z_{t_r}(\xi)]^2)}\bigg\}\Delta
 \bigg)^2 \bigg].
\end{multline*}
An application of H\"{o}lder's inequality
leads then to the estimate
\begin{eqnarray}\nonumber
\lefteqn{\E\le[\sup_{m\leq r < k}|\delta Y_{t_r}|^2  +
\sum_{r=m}^{k-1}|\delta Z_{t_r}|^2\Delta +
\<\delta \tilde M\>_{t_{k}} - \<\delta \tilde M\>_{t_m}
+ \<\delta M\>_{t_{k}} -
\<\delta M\>_{t_m}\ri]}\\ \nonumber
&\leq& C(t_k-t_m) \E\bigg[\max_{m\leq r<k} |\delta Y_{t_r}|^2
+ \sum_{r=m}^{k-1}\bigg\{\le|\delta f^{(\pi)}(t_r,Y^{0}_{t_r},Z_{t_r}^{0},\tilde{Z}_{t_r}^{0})\ri|^2
+  |\delta Z_{t_r}|^2
+ \E_{\nu^{(\pi)}}([\delta \tilde Z_{t_r}(\xi)]^2)\bigg\}\Delta \bigg]\\
&& + 42\,
 \mathbb{E}[|\delta Y_{t_k}|^2]
 \label{fest}
\end{eqnarray}
with $C(u)= 126 K^2\max\{u^2,u\}$ independent of $\pi$.

Next we let $r_0\in(0,T)$ be such that $C(r)\leq \frac{1}{6}\min\{1,C'\}$ for all $r\leq r_0$, where
$C'$ is the constant from Lemma~\ref{finitejumps} [with $\tilde U$ taken equal to the function $\delta\tilde Z^{(\pi)}$].
Let us fix\ $b=[T/r_0]+1$ and consider the regular partition of $[0,T] $ into $b$ intervals. We set for $0\leq \ell \leq b-1,
I_\ell=\{k: t_k\in\pi\cap[\ell T/b,(\ell+1)T/b)\}$, $\ell_*=\min I_\ell$, $\ell^*=\max I_\ell+1$.
Then we obtain from \eqref{fest} and Lemma~\ref{finitejumps} that for every $\ell^*$
\begin{eqnarray*}
\lefteqn{\E\le[\sup_{\ell_*\leq r<\ell^*}|\delta Y_{t_r}|^2  +
\sum_{r=\ell_*}^{\ell^* - 1} |\delta Z_{t_r}|^2\Delta
+
\<\delta \tilde M\>_{t_{\ell^*-1}} - \<\delta \tilde M\>_{t_{\ell_*}}
+\<\delta M\>_{t_{\ell_*}} -
\<\delta M\>_{t_{\ell^*}-1}\ri]}\\
&\leq& 42\cdot\frac{6}{5}\,
 \mathbb{E}\bigg[|\delta Y_{t_\ell^*}|
 + \frac{1}{5}\sum_{r=\ell_*}^{\ell^*-1}\le|\delta f^{(\pi)}(t_r,Y^{0}_{t_r},Z_{t_r}^{0},\tilde{Z}_{t_r}^{0})\ri|^2\Delta\bigg].
\end{eqnarray*}
The proof is completed by a repeated application of this inequality.

\section{Proof of Lemma \ref{tildez}}\label{secproofs}
\begin{proof}[Proof of part (i)]
Recall that
$\tilde{Z}^{(\pi),p}_{t_i}(0)=0$ for all $i$. Thus, to prove
(\ref{lipz}), it is enough to show that
$\tilde{Z}^{(\pi),p}_{t_i}(x)$ is uniformly Lipschitz in
$x\in\mathbb{R}^{d_2}$. Given the assumed form of $F$,
it is possible to find a function
$y^{(\pi),p}_{t_i}:\mathbb{R}^{(d_1+d_2)i}\to \mathbb{R}$ such
that $y^{(\pi),p}_{t_i}(\Delta W^{(\pi)}_{t_0},\Delta
{X}^{(\pi)}_{t_0},\ldots,\Delta W^{(\pi)}_{t_{i-1}},\Delta
{X}^{(\pi)}_{t_{i-1}}):=Y^{(\pi),p}_{t_i}.$
Subsequently, we will suppress
the arguments $\Delta W^{(\pi)}_{t_0},\Delta
{X}^{(\pi)}_{t_0},\ldots,\Delta W^{(\pi)}_{t_{i-2}},\Delta {X}^{(\pi)}_{t_{i-2}}$
whenever there is no ambiguity and write $y^{(\pi),p}_{t_i}(\Delta
W^{(\pi)}_{t_{i-1}},\Delta {X}^{(\pi)}_{t_{i-1}}).$

Fix $t\in[0,T]$ and for every
mesh size $\Delta$ choose $i$ such that $i\Delta\leq t<(i+1)\Delta$ and
denote $w=\Delta w_{t_{i}}$ and $x=\Delta {x}_{t_{i}}.$
Denote
by $Y^{(\pi),p,w,x}_{t_j}$ for $j\geq i+1$ the process
$(Y^{(\pi),p}_{t_j})_{j\geq i+1}$ conditional on $\Delta W^{(\pi)}_{t_{i}}=
w$ and $\Delta {X}^{(\pi)}_{t_{i}} = x.$ For $j\geq i+1$ the
conditioned BS$\Delta$E with solution
$(Y^{(\pi),p,w,x}_t,Z^{(\pi),p,w,x}_t,\tilde{Z}^{(\pi),p,w,x}_t,M^{(\pi),p,w,x}_t)$
can be written as
 \begin{align} \label{condBSDE2}
Y^{(\pi),p,w,x}_{t_j} &= F^{(\pi),w,x}+ \sum_{t_u\geq
t_j}f(t_u,
Y^{(\pi),p-1,w,x}_{t_u},Z^{(\pi),p-1,w,x}_{t_u},\tilde{Z}^{(\pi),p-1,w,x}_{t_u})\Delta-
\sum_{t_u\geq t_j}Z^{(\pi),p,w,x}_{t_u}\Delta W^{(\pi)}_{u}\nonumber\\
& \hspace{0.2cm} -
 \sum_{t_u\geq
t_j} \Big\{\tilde{Z}^{(\pi),p,w,x}_{t_u}(\Delta X^{(\pi)}_{t_i})-\E_{t_{i-1}}[\tilde{Z}^{(\pi),p,w,x}_{t_u}(\Delta X^{(\pi)}_{t_i})]\Big\}
- (M^{(\pi),p,w,x}_{T}-M^{(\pi),p,w,x}_{t_j}),
\end{align}
 with
\begin{align*} F^{(\pi),w,x} &=
H\Big(w_{t_1},\ldots,w_{t_i}+ w,\ldots,w_{t_i} +w+
W^{(\pi)}_{T}-W^{(\pi)}_{t_{i+1}},\\
&\hspace{0.8cm} {x}_{t_1},\ldots,
 {x}_{t_i}+x,\ldots, {x}_{t_i}+
x+{X}^{(\pi)}_{T}-{X}^{(\pi)}_{t_{i+1}}\Big).
\end{align*}

 Clearly,
$y^{(\pi),p}_{t_{i+1}}(w,x)$  has the same law as
$Y^{(\pi),p,w,x}_{t_{i+1}}$. To simplify notation let us assume for
the rest of the proof that $W$ and $X$ (and hence $W^{(\pi)}$ and $X^{(\pi)}$) are
one-dimensional.
The lemma would follow if we could prove through an induction
over $p$ that for every $p$ the following holds: For every $l\in\mathbb{N}_0$
there exist constants
$\tilde{K}_{Y,l,p},\tilde{K}_{Z,l,p},\tilde{K}_{\tilde{Z},l,p}>0$
such that for every $m,k=0,\ldots,l,$ and for all $t:$
\begin{itemize}
\item[(a)] For the mappings $(w,x)\to Y^{(\pi),p,w,x}_j$ we
    have that $\sup_{w,x,\Delta, t_j>
t}\Big|\frac{\partial^{m+k}}{\partial w^m \partial
x^{k}}Y^{(\pi),p,w,x}_{t_j}\Big|\leq \tilde{K}_{Y,l,p}.$
\item[(b)] For the mappings $(w,x)\to
    \tilde{Z}^{(\pi),p,w,x}_{t_j}$ we have that
    $\sup_{w,x,\Delta,t_j> t}\Big|\frac{\partial^{m+k}}{\partial
    w^m \partial
    x^{k}}\tilde{Z}^{(\pi),p,w,x}_{t_j}\Big|\leq
\tilde{K}_{\tilde{Z},l,p}.$
\item[(c)] For the mappings $(w,x)\to
     Z^{(\pi),p,w,x}_{t_j}$ we have that $\sup_{w,x,\Delta,t_j>
t}\Big|\frac{\partial^{m+k}}{\partial w^m \partial
x^{k}} Z^{(\pi),p,w,x}_{t_j}\Big|\leq \tilde{K}_{Z,l,p}.$
\end{itemize}

Notice that (b) implies in particular that $\sup_{w,x,\Delta,t_j>
t}\Big|\tilde{Z}^{(\pi),p,w,x}_{t_j}(x')\Big|\leq
(\tilde{K}_{\tilde{Z},1,p}\vee
\tilde{K}_{\tilde{Z},0,p})(1\wedge |x'|).$

Let us prove (a)---(c). As
$Y^{(\pi),0}=Z^{(\pi),0}=\tilde{Z}^{(\pi),0}=0$, (a)--(c) clearly hold for
$p=0$ with
$\tilde{K}_{Y,l,0}=\tilde{K}_{Z,l,0}=\tilde{K}_{\tilde{Z},l,0}=0$
for all $l$. Now assume that we have shown the induction for
$p-1.$ Let us next show (a)-(c) for $p.$

By the induction assumption for all $j\Delta \geq t$ all higher
derivatives of the processes $Y^{(\pi),p-1,w,x}_j,$
$Z^{(\pi),p-1,w,x}_j,$ and $\tilde{Z}^{(\pi),p-1,w,x}_j$ with respect
to $w$ and $x$ satisfy (a)---(c). As by assumption also all higher
derivatives of $f^{(\pi)}(t_j,\cdot,\cdot,\cdot)$ are bounded as
well uniformly in $t$, $j$ and $\Delta$ with $t_j> t$ we have that
$$\frac{\partial^{m+k}}{\partial w^m \partial x^{k}}f^{(\pi)}(t_j,
Y^{(\pi),p-1,w,x}_{t_j},Z^{(\pi),p-1,w,x}_{t_j},\tilde{Z}^{(\pi),p-1,w,x}_{t_j})
$$
is uniformly bounded by a constant, say $\hat{K}_{l,p-1}.$
Now (\ref{condBSDE2}) entails that
\begin{align*} &\sup_{w,x,\Delta,t_j>
t}\Big|\frac{\partial^{m+k}}{\partial
w^m \partial x^{k}}Y^{(\pi),p,w,x}_{t_j}\Big|\\
& = \sup_{w,x,\Delta, t_j>
t}\bigg|\E_{t_j}\left[\frac{\partial^{m+k}}{\partial w^m \partial x^{k}} F^{(\pi),w,x}+
\sum_{u\geq j}\frac{\partial^{m+k}}{\partial w^{m} \partial x^{k}}f^{(\pi)}(t_u,
Y^{(\pi),p-1,w,x}_{t_u},Z^{(\pi),p-1,w,x}_{t_u},\tilde{Z}^{(\pi),p-1,w,x}_{t_u})\Delta\right]\bigg|
\\&
\leq \sup_{w,x,\Delta}\Big|\Big|\frac{\partial^{m+k}}{\partial w^m \partial x^{k}}
F^{(\pi),w,x}\Big|\Big|_\infty+T\sup_{w,x,\Delta,t_j>
t}
\Big|\Big|\frac{\partial^{m+k}}{\partial w^m \partial x^{k}}f^n((j+1)/n,
Y^{(\pi),p-1,w,x}_{t_j},Z^{(\pi),p-1,w,x}_{t_j},\tilde{Z}^{(\pi),p-1,w,x}_{t_j})\Big|\Big|_\infty
\\
&\leq \tilde{K}_{H,l}+T\hat{K}_{l,p-1}=:\tilde{K}_{Y,l,p},
\end{align*}
where $\tilde{K}_{H,l}$ is the uniform bound of the derivatives
of the function $H$ up to order $l$ for every $l\in \mathbb{N}_0$. This shows that (a) holds.
The validity of (b)
follows immediately from that of (a) and the form \eqref{z2} of $\tilde Z$.

To see that (c) holds true note that for every $t_j> t$
\begin{align*}
|&\frac{\partial^{m+k}}{\partial w^m \partial x^{k}}Z^{(\pi),p,w,x}_{t_j}|
\\
&= \Delta^{-1}\Big|
\E_{t_j}\left[\frac{\partial^{m+k}}{\partial w^m \partial x^{k}}
Y^{(\pi),p,w,x}_{t_j}\Delta W^{(\pi)}_{t_j}\right] \Big|
\\
&= \Delta^{-1} \bigg| \mathbb{E}_{t_j}
\bigg[\Big( \frac{\partial^{m+k}}{\partial w^m \partial x^{k}}y^{(\pi),p,w,x}_{t_{j}}(\Delta W^{(\pi)}_{t_j},\Delta {X}^{(\pi)}_{t_j})
-\frac{\partial^{m+k}}{\partial w^m \partial x^{k}} y_{t_j}^{(\pi),p,w,x}(0,\Delta{X}^{(\pi)}_{t_j})\Big)
\Delta W^{(\pi)}_{t_j}\bigg]\bigg|
\\
&\leq \Delta^{-1} \mathbb{E}_{t_j}\bigg[\Big|
\frac{\partial^{m+k}}{\partial w^m \partial x^{k}}y^{(\pi),p,w,x}_{t_{j}}(\Delta W^{(\pi)}_{t_j},\Delta {X}^{(\pi)}_{t_j})
- \frac{\partial^{m+k}}{\partial w^m \partial x^{k}}y_{t_j}^{(\pi),p,w,x}(0,\Delta {X}^{(\pi)}_{t_{j}})\Big| |
\Delta W^{(\pi)}_{t_{j+1}}|\bigg]\\
&\leq \Delta^{-1} \tilde{K}_{Y,l+1,p} \E\left[
\big|\Delta W^{(\pi)}_{t_{j}}\big|'\,\,|\Delta W^{(\pi)}_{t_j}\big|\right]
=  \Delta^{-1} \tilde{K}_{Y,l+1,p}\E\left[\Big|\Delta W^{(\pi)}_{t_j}\Big|^2\right]=\tilde{K}_{Y,l+1,p}.
\end{align*}
This establishes that (c) holds with $\tilde{K}_{Z,l,p}:=\tilde{K}_{Y,l+1,p}.$
The proof of the induction is complete.
\end{proof}

\begin{proof}[Proof of part (ii).]
The proof is analogous to the proof of part (i).
Denote by $Y^{p,\infty ,x}_{s}$ for $s\geq t$ the process $(Y^{p,\infty}_{t})_{s\geq
t}$ conditional on $\Delta {X}_t = X_t - X_{t^-} = x$.
For $s\geq t$ the
conditioned BSDE with solution
$(Y^{p,\infty ,x}_s,Z^{p,\infty ,x}_s,\tilde{Z}^{p,\infty ,x}_s)$ can be written as
 \begin{align*}
Y^{p,\infty,x}_{s} &= F^{x}+ \int_s^T f(u,
Y^{p-1,\infty,x}_{u},Z^{p-1,\infty,x}_{u},\tilde{Z}^{p-1,\infty,x}_{u})\td u-\int_s^T Z^{p,\infty,x}_u \td W_u
\\
&\hspace{2cm}-
\int_{(s,T]\times\mathbb{R}^{d_2}\setminus\{0\}}\tilde{Z}^{p,\infty,x}_u(z)\tilde{N}_p(\td u\times \td z).
\end{align*}
One may check directly that we have $\tilde{Z}^{p,\infty,x}_t(z)=
Y^{p,\infty,x+z}_t-Y^{p,\infty,x}_t,$ so
that we only have to show that
$\frac{\partial Y^{p,\infty,x}_t}{\partial x}$ is uniformly bounded.
Using the assumptions on $H$ (Assumption~\ref{as:smooth})
this follows by a line of reasoning that is analogous to the one followed in
part (i).
\end{proof}

\end{document}